\documentclass[leqno, a4paper]{article}

\usepackage[T1]{fontenc}
\usepackage[utf8]{inputenc}
\usepackage{booktabs}
\usepackage{multirow}
\usepackage[dvipsnames]{xcolor}
\usepackage{amsmath}
\usepackage{mathtools, amssymb, bbm, amsthm, xparse}

\usepackage{enumitem}

\usepackage{geometry}
\usepackage{changepage}
\usepackage{mathrsfs}
\usepackage{nicefrac}
\usepackage{amssymb}
\usepackage{mathabx}
\usepackage{todonotes}
\usepackage{accents}

\usepackage[sort&compress,numbers]{natbib}

\numberwithin{equation}{section}

\newtheorem{theo}{Theorem}[section]
\newtheorem{prop}[theo]{Proposition}
\newtheorem{lema}[theo]{Lemma}

\newtheorem{coro}[theo]{Corollary}
\theoremstyle{definition}
\newtheorem{Question}{Question}
\newtheorem{rema}[theo]{Remark}
\newtheorem{exam}[theo]{Example}
\newtheorem{defi}[theo]{Definition}

\usepackage{environ}
\NewEnviron{eqs}{
	\begin{equation}\begin{split}
	\BODY
	\end{split}\end{equation}
}
\NewEnviron{eqsn}{
	\begin{equation*}
	\begin{split}
	\BODY
	\end{split}
	\end{equation*}
}

\NewDocumentCommand{\tos}{O{=} O{a.s.}}{\stackrel{\smash{\scriptscriptstyle\mathrm{#2}}}{#1}}
\DeclarePairedDelimiter{\br}{(}{)}

\NewDocumentCommand{\Prob}{O{what} O{\left(} O{\right)}}{\mathbb{P}#2 #1 #3}
\newcommand{\Pro}[1][P]{\mathbb{#1}}
\newcommand{\E}{{\mathbb E}}

\NewDocumentCommand{\tosy}{O{=} O{a.s.}}{\stackrel{\smash{\scriptscriptstyle\mathrm{#2}}}{#1}}

\NewDocumentCommand{\pof}{O{} O{} O{P}}{\mathbb{#3}_{#2}\left(#1\right)}
\NewDocumentCommand{\cmv}{O{X} O{Y} O{}}{\E_{\mathbb{#3}}\left(#1|#2\right)}
\newcommand{\iof}[1][]{\mathbbm{1}_{#1}}

\newcommand{\mb}[1]{\mathbf{#1}}
\newcommand{\mc}[1]{\mathcal{#1}}
\NewDocumentCommand{\lo}{O{x} O{2}}{\log_{#2}\left(#1\right)}
\NewDocumentCommand{\h}{O{X} O{}}{\mathbf{H}_{#2}\left(#1\right)}
\NewDocumentCommand{\ph}{O{X} O{}}{\mathbf{H}_{#2}\left(\mb{#1}\right)}
\NewDocumentCommand{\ch}{O{X} O{Y} O{}}{\mathbf{H}_{#3}\left(#1\;|\; #2\right)}
\NewDocumentCommand{\cph}{O{X} O{Y}}{\mathbf{H}\left(\mb{#1}\;|\; \mb{#2}\right)}
\NewDocumentCommand{\nof}{O{1} O{in}}{\#_{#1}\left(#2\right)}
\NewDocumentCommand{\rh}{O{X} O{Y}}{\mathbf{H}\left(#1 || #2\right)}
\NewDocumentCommand{\rhp}{O{X} O{Y}}{\mathbf{H}\left(\mathbf{#1} || \mathbf{#2}\right)}
\NewDocumentCommand{\proc}{O{X} O{\Z}}{\left(#1_{i}\right)_{i\in #2}}
\NewDocumentCommand{\crh}{O{X} O{Y} O{Z} O{W}}{\mathbf{H}\left(#1 | #2\Big|\Big| #3 | #4\right)}

\NewDocumentCommand{\ts}{O{past}}{\mathcal{T}_{#1}}
\NewDocumentCommand{\nsf}{O{0} O{}}{\mathcal{F}_{#1}^{#2}} 
\NewDocumentCommand{\norm}{O{p} O{TV}}{\left\|#1\right\|_{#2}}

\newcommand{\R}{{\mathbb{R}}}
\newcommand{\Z}{{\mathbb{Z}}}
\newcommand{\N}{{\mathbb{N}}}
\newcommand{\mob}{\boldsymbol{\mu}}
\newcommand{\raz}{\mathbbm{1}}

\usepackage{tikz}
\usetikzlibrary{positioning}
\usetikzlibrary{matrix}
\tikzset{ 
	table/.style={
		matrix of nodes,
		row sep=-\pgflinewidth,
		column sep=-\pgflinewidth,
		node distance=0,
		outer sep=0,
		inner sep=0,
		nodes={
			rectangle,
			draw=black,
			align=center
		},
		minimum height=1.5em,
		text depth=0.5ex,
		text height=2ex,
		nodes in empty cells,
		every even row/.style={
			nodes={fill=gray!20}
		},
		column 1/.style={
			nodes={text width=10em,font=\bfseries}
		},
		column 2/.style={
		nodes={text width=14em}
		},
		column 3/.style={
			nodes={text width=15em}
		},
		row 1/.style={
			nodes={
				fill=black,
				text=white,
				font=\bfseries
			}
		}
	}
}

\usepackage[draft=false]{hyperref}
\hypersetup
{
	colorlinks = true,
	linkcolor = {red},
	anchorcolor = {black},
	citecolor ={orange}, 
	filecolor ={cyan},
	menucolor= {red},
	runcolor ={cyan - same as file color},
	urlcolor= {magenta},
}

\title{Entropy rate of product of independent processes}
\author{Joanna Ku\l{}aga-Przymus \and Micha\l{} D.\ Lema\'{n}czyk}
\begin{document}
\maketitle
\thispagestyle{empty}

\abstract{
We study the multiplicative version of the classical Furstenberg's filtering problem, where instead of the sum $\mb{X}+\mb{Y}$ one considers the product $\mb{X}\cdot \mb{Y}$ ($\mb{X}$ and $\mb{Y}$ are bilateral, real, finitely-valued, stationary independent processes, $\mb{Y}$ is taking values in $\{0,1\}$). We provide formulas for $\ch[\mb{X}\cdot\mb{Y}][\mb{Y}]$. As a consequence, we show that if $\h[\mb{X}]>\h[\mb{Y}]=0$ and $\mb{X}\amalg \mb{Y}$, then $\h[\mb{X}\cdot \mb{Y}]<\h[\mb{X}]$ (and thus $\mb{X}$ cannot be filtered out from $\mb{X}\cdot\mb{Y}$) whenever $\mb{X}$ is not bilaterally deterministic, $\mb{Y}$ is ergodic and $\mb{Y}$ first return to $1$ can take arbitrarily long with positive probability. On the other hand, if almost surely $\mb{Y}$ visits $1$ along an infinite arithmetic progression of a fixed difference (with possibly some more visits in between) then we can find $\mb{X}$ that is not bilaterally deterministic and such that $\h[\mb{X}\cdot\mb{Y}]=\h[\mb{X}]$. As a consequence, a $\mathscr{B}$-free system $(X_\eta,S)$ is proximal if and only if there is always an entropy drop $h(\kappa\ast\nu_\eta)<h(\kappa)$ for any $\kappa$ corresponding to a non-bilaterally deterministic process of positive entropy. These results partly settle some open problems on invariant measures for $\mathscr{B}$-free systems.

\tableofcontents

\section{Background and main results}

In this paper we concentrate on two seemingly unrelated areas:
\begin{enumerate}[label=(\Alph*)]
\item
multiplicative version of the classical Furstenberg's problem on defiltering a noisy signal,
\item
open questions related to invariant measures for so-called $\mathscr{B}$-free systems.
\end{enumerate}
We will now give some background on both, (A) and (B). Then we will present the main technical result and its consequences. Finally, since the paper is a mixture of probabilistic and ergodic tools, we present in a separate section a dictionary allowing for a simultaneous use of both. The remainder of the paper is devoted to the proofs and examples illustrating our results. In the appendix we give some more detailed comments on $\mathscr{B}$-free systems that can be of an independent interest.

        \subsection{Furstenberg's filtering problem}\label{se:fur}
           The classical Furstenberg's filtering problem from the celebrated paper~\cite{MR0213508} concerns two stationary real processes: $\mb{X}$ (the emitted signal) and $\mb{Y}$ (the noise), with $\mb{X}\amalg\mb{Y}$, and the following question is asked:
           \begin{Question}[\cite{MR0213508}]\label{QFu}
              When is $\mb{X}$ measurable with respect to $\sigma$-algebra generated by $\mb{X}+\mb{Y}$? In other words, when it is possible to recover $\mb{X}$ from the received signal $\mb{X}+\mb{Y}$?
           \end{Question}

           In order to address this problem, Furstenberg~\cite{MR0213508} introduced the notion of {disjointness} of dynamical systems, which even today remains one of the central concepts in ergodic theory. Recall that measure-theoretic dynamical systems $(X,\mathcal{B},\mu,T)$ and $(Y,\mathcal{C},\nu,S)$ are \textit{disjoint} if the product measure $\mu\otimes \nu$ is the only $(T\times S)$-invariant measure, projecting as $\mu$ and $\nu$ onto the first and second coordinates, respectively.\footnote{The ergodicity of $(X,\mathcal{B},\mu,T)$ or $(Y,\mathcal{C},\nu,S)$ is a necessary condition for disjointness.} Recall also that each measure-theoretic dynamical system $(X,\mathcal{B},\mu,T)$ yields a family of bilateral, real, stationary processes in the following way. For any measurable function $f\colon X\to \R$, the process $\mb{X}=(f\circ T^i)_{i\in\Z}$ is stationary. In particular, each measurable partition of $X$ into finitely many pieces yields a finitely-valued stationary process. On the other hand, each real stationary process $\mb{X}$ yields a (symbolic) measure-theoretic dynamical system by taking the left shift $S$ on the product space~$\R^\Z$, with the invariant measure given by the distribution of~$\mb{X}$ (if the state space of $\mb{X}$ is smaller than $\R$, we can consider the left shift $S$ on the appropriate smaller product space). A crucial basic observation is that whenever the family of functions $\{f\circ T^i:i\in \Z\}$ generates $\mathcal{B}$ then the resulting symbolic (measure-theoretic) dynamical system is isomorphic to $(X,\mathcal{B},\mu,T)$. Last, but not least, we say that processes $\mb{X}$ and $\mb{Y}$ are \textit{absolutely independent}, whenever the resulting dynamical systems are disjoint. Furstenberg showed that absolute independence is a sufficient condition, under which one has the positive answer to Question~\ref{QFu}:
           \begin{theo}[\cite{MR0213508}]\label{tw:fu}
              Suppose that $\mb{X}$ and $\mb{Y}$ are integrable and that $\mb{X}$ is absolutely independent from $\mb{Y}$. Then $\mb{X}$ is measurable with respect to $\sigma$-algebra generated by $\mb{X}+\mb{Y}$.
           \end{theo}
           Moreover, Garbit~\cite{MR2783975} showed that the integrability assumption can be dropped and the assertion of Theorem~\ref{tw:fu} still holds.

           We are interested in the following modification of Question~\ref{QFu}: instead of the sum of processes $\mb{X}$ and $\mb{Y}$, we consider their product
           $$
              \mb{M}:=\mb{X}\cdot\mb{Y}=(X_i\cdot Y_i)_{i\in\Z}.
           $$
           Notice that if $\mb{X}$ and $\mb{Y}$ take only positive values, we can define processes $\log \mb{X}$ and $\log \mb{Y}$. Since $\log\mb{M}=\log\mb{X}+\log \mb{Y}$, by the result of Garbit, $\mb{X}$ can be recovered from $\mb{M}$ whenever $\mb{X}$ and $\mb{Y}$ are disjoint. Therefore, it is natural to ask whether the same conclusion as in Theorem~\ref{tw:fu} holds for processes that admit zero as a value. The simplest instance of this is when the state space, e.g., of $\mb{Y}$, equals $\{0,1\}$. One can think of $\mb{M}$ as of the original signal $\mb{X}$, where some of the information was lost (due to $Y_i=0$), instead of just being perturbed (by adding $Y_i$ to $X_i$). Thus, we deal with the following problem:
           \begin{Question}\label{Q2a}
              Let $\mb{X}$ and $\mb{Y}$ be bilateral, real, finitely-valued, stationary processes, with $Y_i\in \{0,1\}$. Suppose that $\mb{X}\amalg \mb{Y}$. Is it possible to recover  $\mb{X}$ and / or $\mb{Y}$ from $\mb{M}$?
           \end{Question}

A similar (in fact, much more general) problem of retrieving a lost signal was studied by Furstenberg, Peres and Weiss in~\cite{Fu-Pe-We}. Let $\mb{X}^{(i)}=\left( X_i^{(U_i)}\right)_{i\in \mathbb{Z}}$, where $i\in\N$, be a family of processes and $\mb{U}$ be an $\N$-valued process. Suppose that all these processes are stationary and define
\[
\mb{X}^{(\mb{U})}=\left(X_i^{(U_i)} \right)_{i\in\Z}
\]
(informally, $\mb{U}$ chooses among the family of processes). 
\begin{Question}\label{FPW}
Is it possible to recover  $\mb{U}$ from $\mb{X}^{(\mb{U})}$?
\end{Question}
In order to answer this question the authors of \cite{Fu-Pe-We} introduce the notion of double disjointness. We say that process $\mb{Y}$ is doubly disjoint from $\mb{X}$ if every self-joining of $\mb{Y}$ is absolutely disjoint from $\mb{X}$. In other words if $(\mb{X}',\mb{Y}', \mb{Y}'')$ is a stationary process such that $\mb{X}' \sim \mb{X}$ and $\mb{Y}', \mb{Y}'' \sim \mb{Y}$ then $\mb{X}'\amalg(\mb{Y}', \mb{Y}'') $. The most basic example of doubly disjoint processes arises when $\mb{Y}$ is of zero entropy rate (then every self-coupling of $\mb{Y}$ has zero entropy) and $\mb{X}$ has trivial tail-$\sigma$-field (let us add that in fact if $\mb{Y}$ is doubly disjoint from $\mb{X}$ then \textbf{necessarily} $\ph[Y] = 0$ and $\mb{X}$ is ergodic). (For the definition of entropy rate, see~\eqref{zwykla} below.) Now, the main result of~\cite{Fu-Pe-We} can be summarized (roughly) as follows. Suppose that $\mb{X}^{(i)}$ for $i\in\N$ and $\mb{U}$ are jointly stationary. If $\mb{U}$ is doubly disjoint from each $\mb{X}^{(i)}$ for $i\in\N$ then one can retrieve $\mb{U}$ from $\mb{X}^{(\mb{U})}$. 
		  		
		  		Let us explain how to fit this theorem to our setting from Question \ref{Q2a} (and retrieve $\mb{Y}$ from $\mb{M}$). Consider two processes $\mb{X}^{(i)}$, for $i \in \{0, 1\}$, where
		  		\begin{equation}
			  		X_j^{(i)} = iX_j 
		  		\end{equation}
		  		and take $\mb{U} = \mb{Y}$. Then $\mb{X}^{(\mb{U})} = \mb{X} \cdot \mb{Y}$ and the theorem states that we can retrieve $\mb{Y}$ from  $\mb{X} \cdot \mb{Y}$ as soon as $\mb{Y}$ is doubly disjoint from $\mb{X}$. Since the role of $\mb{X}$ and $\mb{Y}$ is here not symmetric (and $\mb{M}$ and $\mb{Y}$ do not determine $\mb{X}$ unlike when one studies the sum $\mb{X}+\mb{Y}$), it is interesting to ask whether one can also retrieve~$\mb{X}$. To stay compatible with the notion of double disjointness, we will assume that $\h[\mb{X}]>\h[\mb{Y}]=0$.
           Then, clearly, a necessary condition for having the positive answer to Question~\ref{Q2a} is that $\h[\mb{M}]=\h[\mb{X}]$. Having this in mind, we will deal with the following three more specific problems:
           \begin{Question}\label{Q:procesy}$~$
              \begin{enumerate}[label={(\Alph*})]
                 \item\label{qa1}
                       Is there a general formula for the entropy rate $\mb{H}(\mb{M})$ of $\mb{M}=\mb{X}\cdot \mb{Y}$?
                 \item\label{qa2}
                       Do we always have $\mb{H}(\mb{M})>0$ whenever $\mb{H}(\mb{X})>0$?
                 \item\label{qa3}
                       Can we have $\mb{H}(\mb{M})=\mb{H}(\mb{X})$ with $\h[\mb{X}]>0$?
              \end{enumerate}
           \end{Question}
           \begin{rema}
           Notice that the answers to Question~\ref{QFu} in~\cite{MR0213508} and to Question~\ref{FPW} in~\cite{Fu-Pe-We} depend only on the properties of the underlying dynamical systems corresponding to $\mb{X}$ and $\mb{Y}$. In this paper the situtation will be different and the ability to defilter $\mb{X}$ from $\mb{M}$ will highly depend on the properties of the stochastic processes under consideration. Cf.\ Example~\ref{i}.
           \end{rema}

        \subsection{Invariant measures for \texorpdfstring{$\mathscr{B}$}{}-free systems}\label{se:inv}

           Question~\ref{Q:procesy} is a	 generalization of some questions asked in~\cite{Ku-Le-We} in the context of $\mathscr{B}$-free systems. For $\mathscr{B}\subset \N\setminus \{1\}$, consider the corresponding sets of multiples and $\mathscr{B}$-free integers, i.e.\
           $$
              \mathcal{M}_\mathscr{B}:=\bigcup_{b\in\mathscr{B}}b\Z \text{ and }\mathcal{F}_\mathscr{B}:=\Z\setminus \mathcal{M}_\mathscr{B},
           $$
           respectively. Such sets were studied already in the 30's from the number-theoretic viewpoint (see, e.g.~\cite{Davenport:1933aa,zbMATH03014412,MR1512943,Davenport1936,MR0043835,MR1574879}). The most prominent example of $\mathcal{F}_\mathscr{B}$ is the set of square-free integers (with $\mathscr{B}$ being the set of squares of all primes). The dynamical approach to $\mathscr{B}$-free sets was initiated by Sarnak~\cite{sarnak-lectures} who proposed to study the dynamical system given by the orbit closures of the M\"obius function $\mob$ and its square $\mob^2$ under the left shift $S$ in $\{-1,0,1\}^\Z$.\footnote{Recall that $\mob(n)=(-1)^k$ if $n$ is a product of $k$ distinct primes, $\mob(1)=1$ and $\mob(n)=0$ otherwise; $\mob^2$ is the characteristic function of the set of square-free integers.} For an arbitrary $\mathscr{B}\subset \N\setminus \{1\}$, let $X_\eta$ be the orbit closure of $\eta=\raz_{\mathcal{F}_\mathscr{B}}\in \{0,1\}^\Z$ under the left shift, i.e.\ we deal with a \textit{subshift} of $(\{0,1\}^\Z,S)$.\footnote{More generally, given a finite alphabet $\mc{X}$, we define $S$ to be the left shift on $\mc{X}^\Z$, i.e.\ $S(\proc[x])=\proc[y]$, where $y_i=x_{i+1}$, $i\in\Z$. Each closed $S$-invariant subset of $\mc{X}^\Z$ is called a \textit{subshift}. } We say that $(X_\eta,S)$ is a \textit{$\mathscr{B}$-free system}.
           In the so-called Erd\"os case (when the elements of $\mathscr{B}$ are pairwise coprime, $\mathscr{B}$ is infinite and $\sum_{b\in\mathscr{B}}1/b<\infty$), $X_\eta$ is \textit{hereditary}: for $y\leq x$ coordinatewise, with $x\in X_\eta$ and $y\in\{0,1\}^\Z$, we have $y\in X_\eta$. In other words, $X_\eta=M(X_\eta\times \{0,1\}^\Z)$, where $M$ stands for the coordinatewise multiplication of sequences. For a general $\mathscr{B}\subset \N\setminus\{1\}$,  ${X}_\eta$ may no longer be hereditary and we consider
           its hereditary closure $\widetilde{X}_\eta:=M(X_\eta\times \{0,1\}^\Z)$ instead. Usually, one assumes at least the \textit{primitivity} of $\mathscr{B}$ (i.e.\ $b\ndivides b'$ for $b\neq b'$ in $\mathscr{B}$).

           Given a \textit{topological dynamical system} $(X,T)$, i.e.\ a homeomorphism $T$ acting on a compact metric space $X$, let $\mathcal{B}$ be the $\sigma$-algebra of Borel subsets of $X$. By $\mathcal{M}(X,T)$ we will denote the set of all probability Borel $T$-invariant measures on $X$ and $\mathcal{M}^e(X,T)$ will stand for the subset of ergodic measures. Each choice of $\mu\in \mathcal{M}(X,T)$ results in a \textit{measure-theoretic dynamical system}, i.e.\ a 4-tuple $(X,\mathcal{B},\mu,T)$, where $(X,\mathcal{B},\mu)$ is a standard probability Borel space, with an automorphism $T$. We often skip $\mathcal{B}$ and write $(X,\mu,T)$. Recall also that $x\in X$  is said to be \textit{generic} for $\mu\in\mathcal{M}(X,T)$, whenever $\lim_{N\to\infty}\frac1N\sum_{n\leq N}\delta_{T^nx}=\mu$ in the weak topology. If the convergence takes place only along a subsequence $\br*{N_k}_{k\geq 1}$ then we say that $x$ is \textit{quasi-generic} for $\mu$. Each measure $\mu$ resulting in this way yields a measure theoretic dynamical system $(X,\mathcal{B},\mu,T)$.

           A central role in the theory of $\mathscr{B}$-free systems is played by the so-called \textit{Mirsky measure}, denoted by $\nu_\eta$. In the Erd\"os case, $\eta$ is a generic point for $\nu_\eta$ (in general, $\eta$ is quasi-generic along some natural sequence $(N_k)$), see~\cite{Dy-Ka-Ku-Le}. It was shown in \cite{Ku-Le-We,Dy-Ka-Ku-Le} that all invariant measures for $\widetilde{X}_\eta$ are of the following special form:
           \begin{theo}[cf.\ Section~\ref{se:a2}]\label{tw:postac}
              For any $\nu\in\mathcal{M}(\widetilde{X}_\eta,S)$, there exists $\rho\in \mathcal{M}(X_\eta\times \{0,1\}^\Z, S\times S)$ such that $\rho|_{X_\eta}=\nu_\eta$ and $M_\ast (\rho)=\nu$.\footnote{By $\rho|_{X_\eta}$, we denote the projection of $\rho$ onto the first coordinate $X_\eta$; $M_\ast(\rho)$ stands for the image of $\rho$ via $M$. We will use similar notation later on, too.}
           \end{theo}

           Recall that given a measure-theoretic dynamical system $(X,\mathcal{B},\mu, T)$, any $T$-invariant sub-$\sigma$-algebra $\mathcal{A}\subset\mathcal{B}$ is called a \textit{factor} of $(X,\mathcal{B},\mu,T)$.\footnote{Equivalently, if $\pi\colon (X,\mathcal{B},\mu,T) \to (Z,\mathcal{D},\rho,R)$ intertwines the actions of $T$ and $R$, then $R$ is called a factor of $T$ (as $\mathcal{A}=\pi^{-1}(\mathcal{D})\subset \mathcal{B}$ is $T$-invariant).}  Notice that given $\nu$ and $\rho$ as in Theorem~\ref{tw:postac}, $(\widetilde{X}_\eta,\nu,S)$ is a factor of $(X_\eta\times \{0,1\}^\Z,\rho,S\times S)$.

           The measure-theoretic entropy of $(X,\mathcal{B},\mu,T)$ will be denoted by $h_\mu(T,\mathcal{B})$. If no confusion arises, we will also write $h(\mu,T)$ or even $h(\mu)$. If $\mb{X}$ is a finitely-valued stationary process determining $(X,\mu,T)$ (as described in Section~\ref{se:fur}) then $\h[\mb{X}]=h(\mu)$.

           The Mirsky measure $\nu_\eta$ is of zero entropy. Moreover, it was shown in~\cite{Ku-Le-We} in the Erd\"os case that $({X}_\eta,S)$ is intrinsically ergodic (it has exactly one measure realizing the topological entropy). Its measure of maximal entropy equals $M_\ast (\nu_\eta\otimes B_{\nicefrac12,\nicefrac12})$, where $B_{\nicefrac12,\nicefrac12}$ stands for the Bernoulli measure on $\{0,1\}^\Z$ of entropy $\log 2$. These results were extended in~\cite{Dy-Ka-Ku-Le} to a general $\mathscr{B}$ (one needs to replace $X_\eta$ with $\widetilde{X}_\eta$). In the Erd\"os case, the topological entropy of $({X}_\eta,S)$ is equal $\prod_{i\geq 1} (1-\nicefrac1{b_i})$\footnote{In the special case of the square-free system the result was proved by Peckner~\cite{MR3430278}.} (in general, the topological entropy of $(\widetilde{X}_\eta,S)$ equals $\overline{d}(\mathcal{F}_\mathscr{B})$~\cite{Dy-Ka-Ku-Le}).\footnote{By $\overline{d}$ we denote the upper asymptotic density.} This led to the study of \textit{product type measures} (or \textit{multiplicative convolutions}):
           $$
              \nu_\eta \ast \kappa:=M_\ast (\nu_\eta\otimes \kappa).
           $$
           In particular, it was proved that
           \[
              0<h(\nu_\eta\ast B_{\nicefrac12,\nicefrac12})<h(B_{\nicefrac12,\nicefrac12}).
           \]
           Moreover, it was shown that for each value $0\leq h\leq \prod_{i\geq 1} (1-\nicefrac1{b_i})$ there is an ergodic measure $\kappa$ satisfying $h(X_\eta,\nu_\eta\ast \kappa)=h$. However, some fundamental questions related to such measures were left open -- they turn out to be a special instance of Question~\ref{Q:procesy} (see Question 1 in~\cite{Ku-Le-We}):
           \begin{Question}\label{Q1}
              $~$
              \begin{enumerate}[label={(\Alph*})]
                 \item\label{a}
                       Is there a general formula for the entropy $h(\nu_\eta\ast \kappa)$ of $\nu_\eta\ast \kappa$?
                 \item\label{b}
                       Do we always have $h(\nu_\eta\ast \kappa)>0$ whenever $h(\kappa)>0$?
                 \item\label{c}
                       Can we have $h(\nu_\eta\ast\kappa)=h(\kappa)$ with $h(\kappa)>0$?
              \end{enumerate}
           \end{Question}

    \subsection{Main technical result}\label{se11}
    
	Our main tool used to answer Questions~\ref{Q:procesy} and \ref{Q1} is concerned with the entropy rate of stationary processes. Before we can formulate it, we need some definitions and notation that will be used througout the whole paper.

        All random variables and processes will be defined on a fixed probability space $\br*{\Omega, \mc{F},\Pro}$. Sometimes, we will replace the underlying probability measure $\Pro$ by its conditioned version, $\Pro_A(\cdot) = \pof[\cdot \cap A] / \pof[A]$, where  $A\in \mc{F}$ with $\Pro(A)>0$. In particular, $\E_A$ will stand for the expectation taken with respect to $\Pro_A$. For convenience sake, we will write $A, B$ instead of $A\cap B$ for any $A,B\in\mc{F}$: for example, $\E_{A, B}$ stands for $\E_{A\cap B}$. A central role will be played by the \textbf{Shannon entropy} of a random variable $X$, denoted by us by $\h[X]$. Although we will recall basic definitions and properties related to $\h[X]$, some well-known facts will be taken for granted (all of them can be found in \cite{MR3134681}).  All random processes will be \textbf{bilateral} and \textbf{real}. Usually, they will be also \textbf{finitely-valued} and \textbf{stationary}, however sometimes we will need auxiliary countably-valued, non-stationary processes. Recall that a process $\mb{X} = \proc$ is stationary if $\proc$ has the same distribution as $\left(X_{i + 1}\right)_{i\in\Z}$ and finitely-valued if, for every $i\in\Z$, $X_i \in \mc{X}$,  with $\left|\mc{X}\right| < \infty$.

        Let now  $X, Y$ be random variables taking values in finite state spaces $\mc{X}$, $\mc{Y}$ respectively and fix $A\in\mc{F}$ with $\pof[A] > 0$. We put $\h[X][A] = -\sum_{x\in\mc{X}} \pof[X = x][A]\log_2\pof[X = x][A]$. Moreover,  $\ch[X][Y][A] = \sum_{y\in\mc{Y}}\pof[Y = y][A]\h[X][Y= y, A]$ will stand for the \textit{conditional Shannon entropy} of $X$ with respect to $Y$. When $\pof[A] = 1$, we will omit subscript $A$ and write $\h[X]$ and $\ch[X][Y]$, respectively.

        To shorten the notation, we will use the following convention. For a subset $A = \left\{i_1, \ldots, i_n\right\} \subset \Z$ with $i_1 < i_2 < \cdots < i_n$ and a process $\mb{X} = \proc$, we will write
        \begin{equation*}
           X_A = \left(X_{i_1}, X_{i_2},\ldots, X_{i_n}\right).
        \end{equation*}
        Moreover, for any $k \le  \ell$ in $\Z$, we define \textit{integer intervals}:
        \begin{equation*}
           [k, \infty) := \left\{k, k + 1, \ldots\right\},  \qquad  (-\infty, \ell] := \left\{\ell, \ell-1, \ldots \right\}, \qquad [k,\ell] := \left\{k, k + 1, \ldots, \ell\right\}.
        \end{equation*}
        For example, $X_{[0, n]} = \left(X_0, \ldots, X_n\right)$ for $n\in\N$. It is natural and convenient to interpret $[k, \ell]$ as $\emptyset$ if $\ell < k$, $\h[X_{\emptyset}] = 0$ and $\ch[X][Y_{\emptyset}] = \h[X]$.
        
        Consider now two random processes $\mb{X}= \proc[X]$ and $\mb{Y}  = \proc[Y]$ such that $(\mb{X}, \mb{Y}) := \left((X_i, Y_i)\right)_{i\in\Z}$ is stationary. Then
        \begin{equation}\label{zwykla}
           \ph[X] =  \lim\limits_{n\to\infty}\frac{1}{n}\h[X_{[0, n - 1]}], \quad \ch[\mb{X}][\mb{Y}] = \lim\limits_{n \to \infty} \frac{1}{n} \ch[X_{[0, n - 1]}][Y_{[0, n - 1]}]
        \end{equation}
        will denote, respectively, the \textit{entropy rate} of $\mb{X}$ and the \textit{relative entropy rate} of $\mb{X}=\proc[X]$ with respect to $\mb{Y}=\proc[Y]$. By the stationarity of $\mb{X}$, $\ph[X] = \lim\limits_{n\to\infty}\ch[X_0][X_{[-n, -1]}]$. Note that both limits in~\eqref{zwykla} exist due to the subadditivity of appropriate sequences.
        \begin{rema}\label{rem11}
           Sometimes it is convenient to extend the classical definition of the conditional entropy, $\ch[X][Y]$,  to $\ch[X][\mc{G}]$, where $X$ is a finitely-valued random variable and $\mc{G} \subset \mc{F}$ is a sub-$\sigma$-algebra (see \cite{MR1958753}, Chapter $14$,  for a precise construction and proofs). This extension is justified by the following facts. If $\mc{G} =\sigma(Y)$ then $\ch[X][\sigma(Y)] = \ch[X][Y]$ for any random variable $Y$.\footnote{Recall that $\sigma(Y)$ stand for the smallest $\sigma$-algebra making $Y$ measurable.} If $\mc{H}\subset \mc{G}\subset\mc{F}$ are sub-$\sigma$-algebras then $\ch[X][\mc{G}] \le \ch[X][\mc{H}]$. Moreover, if $\mc{G}_n \searrow \mc{G}$ or $\mc{G}_n \nearrow \mc{G}$ then $\ch[X][\mc{G}_n] \nearrow \ch[X][\mc{G}]$ or $\ch[X][\mc{G}_n] \searrow \ch[X][\mc{G}]$, respectively. Thus, for example, it makes sense to write $\ph[X] = \ch[X_0][X_{(-\infty, - 1]}] = \lim\limits_{n\to\infty} \ch[X_0][X_{(-n, - 1]}]$. The chain rule is still valid, namely,  if $X$ and $Y$ are finitely-valued then 
           \begin{equation}\label{chr}
           \ch[X, Y][\mc{G}] =  \ch[X][\mc{G}] + \ch[Y][\sigma(\mc{G}, \sigma(X))].
           \end{equation}
           Furthermore, $\ch[X][\mc{G}] = 0$ if and only if $X$ is $\mc{G}$-measurable and $\ch[X][\mc{G}] = \h[X]$ if and only if $X$ is independent of $\mc{G}$. 
        \end{rema}
        \begin{rema}\label{ram12}
           We will often omit some technicalities concerning events of zero probability. First, we tacitly assume that $\mc{F}$ is complete (i.e.\ all subevents of zero-measure events are measurable). Secondly, when considering sub-$\sigma$-fields associated with random processes, we think of them as of \textbf{measure-$\sigma$-algebras} (intuitively, we look at them "up to events of probability zero"). Given sub-$\sigma$-fields $\mc{G},\mc{H}\subset \mc{F}$, sometimes we will write
           \begin{equation*}\label{definition of measure subset}
              \mc{G}\tos[\subset][\Pro] \mc{H}
           \end{equation*}
           to stress that for every $G \in \mc{G}$ there is $H\in\mc{H}$ such that $\pof[G\triangle H] = 0$ but not necessarily $\mc{G}\subset \mc{H}$ (with obvious modifications for $\tos[\supset][\Pro]$ and $\tos[=][\Pro]$). However, in most cases, we will skip such considerations, cf.\ the last sentence of the previous remark.
        \end{rema}
        Given processes $\mb{X} = \proc$ and $\mb{Y} = \proc[Y]$, we will be interested in the entropy rate $\ph[\mb{X}\cdot \mb{Y}]$ of their product $\mb{X}\cdot \mb{Y} = (X_i \cdot Y_i)_{i\in\Z}$.
        Our \textbf{standing assumptions} (unless stated otherwise) will be that:
        \begin{enumerate}[label={(\roman*)}]
           \item $\mb{X}$ is finitely-valued, $\mb{Y}$ is binary ($Y_i\in \{0,1\}$ for $i\in\Z$) and $\Pro(Y_0=1)>0$, \label{RAZ}
           \item $\mb{X}\amalg \mb{Y}$, i.e.\ $\mb{X}$ and $\mb{Y}$ are independent.\label{DWA}
        \end{enumerate}
        Notice that by the independence of $\mb{X}$ and $\mb{Y}$, process $(\mb{X}, \mb{Y})$ is stationary. Moreover, $\mb{X}\cdot \mb{Y}$ is a factor of $(\mb{X}, \mb{Y})$.\footnote{Cf.\ Remark~\ref{rema:factor}.}
        The quantity $\cph[\mb{X}\cdot \mb{Y}][Y]$ turns out to be easier to deal with than $\ph[\mb{X}\cdot \mb{Y}]$. A particular emphasis will be put on the case when $\h[\mb{Y}]=0$ in which $\ch[\mb{X}\cdot \mb{Y}][\mb{Y}]=\h[\mb{X}\cdot\mb{Y}]$\footnote{To see this, it suffices to notice that given a process $\mb{Z}$, we have $\frac{1}{n}\ch[Z_{[0, n]}][Y_{[0, n]}] \le \frac{1}{n}\h[Z_{[0, n]}] =  \frac{1}{n}\ch[Z_{[0, n]}][Y_{[0, n]}] + \frac{1}{n}\h[Y_{[0, n]}]$. In particular, we can take $\mb{Z}=\mb{X}\cdot \mb{Y}$.} and $\ph[X\cdot Y] \le \ph[X]$.\footnote{Using the fact that factors cannot increase entropy and by the subadditivity of entropy rate, we have $\ph[X\cdot Y]\le \ph[(X, Y)] \le \ph[X] + \ph[Y]$.}

        Let $\mb{R}=\mb{R}(\mb{Y})=\proc[R]$ be the \textit{return process}, i.e.\ the process of consecutive {\it arrival times} of $\mb{Y}$ to 1:
        \begin{equation}\label{procesR}
           R_i=\begin{cases}
              \inf\{j \ge 0 : Y_j = 1\},        & i=0,      \\
              \inf\{j\ge R_{i - 1} : Y_j = 1\}, & i\geq 1,  \\
              \sup\{j < R_{i + 1} : Y_j = 1\},  & i\leq -1.
           \end{cases}
        \end{equation}
        Note that, in general, $\mb{R}$ can be countably-valued. If $\mb{Y}$ is ergodic then it visits 1 infinitely often, both in the future and in the past and, thus, $\mb{R}$ is well-defined almost everywhere. However, we don't need to assume the ergodicity of $\mb{Y}$ to be able to speak of $\mb{R}$ and we will just assume that:
\begin{enumerate}[label={(\roman*)}]
\setcounter{enumi}{2}
\item\label{TRZY}
        $\mb{Y}$ is such that the definition of $\mb{R}$ makes sense.
\end{enumerate}        
Whenever \ref{RAZ}, \ref{DWA} and \ref{TRZY} hold, we will say that the pair $(\mb{X},\mb{Y})$ is {\em good}. If $\mb{Y}$ is binary, with $\mathbb{P}(Y_0=1)>0$ and such that~\ref{TRZY} holds, we will say that $\mb{Y}$ is {\em good}.
        \begin{rema}

           We will use lowercase letters to denote \textit{realizations} of the corresponding random processes (denoted by uppercase letters). Recall that $\mb{x}=\proc[x]$ is a realization of $\mb{X} =\proc$ if there exists $\omega \in \Omega$ such that $x_i = X_i(\omega)$ for all $i\in\Z$. Moreover, we will tacitly assume that $\omega$ belongs to some "good" subset of $\Omega$ of probability $1$. For example, for $\mb{R}$, our standing assumption will be that $\omega$ realizing $\mb{r}$ belongs to the set where $\mb{Y}$ visits $1$ infinitely often in both directions. In general, if some property of a process $\mb{X}$ has probability $1$, then realization $\mb{x}$ inherits it. For example, if we consider $\mb{Y}$ under $\Pro_{Y_0 = 1}$ then every realization $\mb{y}$ will satisfy $y_0 = 1$.
        \end{rema}

        Here is our main technical result:
        \begin{theo}[answer to Question~\ref{Q:procesy}\ref{qa1}]\label{main}
           Suppose that $(\mb{X},\mb{Y})$ is {\em good}. Then
           \begin{enumerate}[label={(\Alph*)}]
              \item
                    $\ch[\mb{X}\cdot\mb{Y}][\mb{Y}]= \pof[Y_0 = 1]\E_{Y_{0} = 1} \ch[X_0][X_{\{r_{-1}, r_{-2}, \cdots\}}]|_{r_{-i}=R_{-i}}$.\label{mainA}
           \end{enumerate}
           If additionally $\mb{Y}$ is ergodic then
           \begin{enumerate}[label={(\Alph*)}]
           \setcounter{enumi}{1}
           \item\label{jeszcze_jeden}
            $\ch[\mb{X}\cdot\mb{Y}][\mb{Y}]= \h[\mathbf{X}] - \Pro(Y_0=1) \E_{Y_0=1} \ch[X_{[1,r_1-1]}][X_{(-\infty,0]}, X_{\{r_1,r_2,\ldots\}}]|_{r_i=R_i}$.         
           \end{enumerate}           
        \end{theo}
        
        \begin{rema}
           The above expectations are to be understood in the following way:
           \begin{itemize}
              \item
                    we compute $\ch[X_0][X_{\{r_{-1}, r_{-2}, \cdots\}}]$ or $\ch[X_{[1,r_1-1]}][X_{(-\infty,0]}, X_{\{r_1,r_2,\ldots\}}]$ for all realizations $\mb{r} = \proc[r]$ thus obtaining a function $f(\mb{r})$ depending on $\mb{r}$;
              \item
                    we find $\E_{Y_0 = 1} f(\mb{R})$.
           \end{itemize}

        \end{rema}

\subsection{Consequences of the main technical result}\label{longer}
    Clearly, Theorem~\ref{main} gives an answer to Questions~\ref{Q:procesy}\ref{qa1} and~\ref{Q1}\ref{a}. We will say now how it is related to Questions~\ref{Q:procesy}\ref{qa2}, \ref{Q1}\ref{b}, \ref{Q:procesy}\ref{qa3} and \ref{Q1}\ref{c}. The details and longer proofs are included in Section~\ref{here_are_proofs}.

   \subsubsection{Answer to Questions~\ref{Q:procesy}\ref{qa2} and \ref{Q1}\ref{b}}
 
    Notice first that
        \[
           \h[\mb{X}]\leq \ch[X_0][X_{\{r_{-1}, r_{-2}, \cdots\}}]\leq \h[X_0]
        \]
        for each choice of negative integers $\dots<r_{-2}<r_{-1}<0$. Therefore, by Theorem~\ref{main}~\ref{mainA}, we obtain immediately the following:
        \begin{coro}[positive answer to Question~\ref{Q:procesy}\ref{qa2}]\label{C4}
           Suppose that $(\mb{X},\mb{Y})$ is {\em good}. Assume additionally that $\h[\mb{Y}]=0$. Then
           $$
              \pof[Y_0=1]\h[\mb{X}]\leq \h[\mb{M}]\leq \pof[Y_0=1]\h[X_0].
           $$
           In particular,
           \begin{equation}\label{answ}
              \h[\mb{M}]>0 \text{ whenever }\h[\mb{X}]>0.
           \end{equation}
        \end{coro}

        \begin{rema}
           The lower bound in Corollary~\ref{C4} is attained for exchangeable processes (see Proposition~\ref{lowerboundforex}), whereas the upper bound is attained for i.i.d.\ processes. If $\mb{X}$ is a Markov chain (which is not i.i.d.), both inequalities are strict, see Section~\ref{LM}.
        \end{rema}
        \begin{rema}[positive answer to Question~\ref{Q1}\ref{b}]
           Implication~\eqref{answ} means, in particular, that the answer to Question~\ref{Q1}\ref{b} is positive, whenever $\nu_\eta\neq\delta_{(\ldots 0,0,0\ldots)}$. In Section~\ref{erg4b} we present an alternative ergodic-theoretic approach to this problem. The proof presented therein is much shorter, on the other hand it addresses directly Question~\ref{Q1}\ref{b}, without providing any explicit formulas.
        \end{rema}

        \begin{rema}\label{rema1.8}
           If one drops the assumption that $\mb{X} \amalg \mb{Y}$ then the situation changes completely and one can get $\h[\mb{M}] = 0$ (with $\h[\mb{X}]>0$ and $\pof[Y_0=1]>0$). To see how far this can go, consider
           \[
              \mb{X} = \mb{Z}\cdot \mb{W} \text{ and }\mb{Y} = \mb{1 - W} = \br*{1 - W_i}_{i \in \Z},
           \]
           where
           \[
              \mb{Z} \amalg \mb{W}, \ph[W] = 0\text{ and }\mathbb{P}(W_0=0)\cdot\mathbb{P}(W_0=1)>0.
           \]
           Then $\mb{M}$ is a trivial zero process, in particular, we have $\h[\mb{M}] = 0$. On the other hand, by Corollary~\ref{C4}, $\ph[X] = \h[\mb{Z}\cdot \mb{W}] > 0  = \h[\mb{W}] = \h[\mb{Y}]$. Cf.\ also Section~\ref{ergo} for more examples of ergodic-theoretic flavour.
        \end{rema}

\subsubsection{Answer to Questions \ref{Q:procesy}\ref{qa3} and \ref{Q1}\ref{c}}
Answers to Questions \ref{Q:procesy}\ref{qa3} and \ref{Q1}\ref{c} are more complex and they are related to the notion of a bilaterally deterministic process.
        \begin{defi}
           We say that a stationary process $\mb{Z} = \left(Z_i\right)_{i \in \Z}$ is \textit{bilaterally deterministic}  if, for all $k\in\N$,
           \begin{equation*}\label{definition bilaterally deterministic process}
              \ch[Z_{[0,k]}][Z_{(-\infty, -1]},Z_{[k+1, \infty)}] = 0.
           \end{equation*}
        \end{defi}
        \begin{rema}\label{rema1.10}
           The notion of a bilaterally deterministic process was introduced by Ornstein and Weiss~\cite{Ornstein_1975}, in terms of the following (double) tail sigma-algebra:
           \[
              \ts[d] := \bigcap_{n\geq1} \sigma\left(Z_{(-\infty, -n]}, Z_{[n, \infty)}\right).
           \]
	  Notice that the following conditions are equivalent:
           \begin{itemize}
              \item
                    $\mb{Z}$ is bilaterally deterministic,
              \item
                    $Z_{[-k,k]}\in \ts[d] \text{ for each }k\geq 1$,
              \item
                    $\sigma(\mb{Z})=\ts[d]$.
           \end{itemize}
           Indeed, e.g., if $\mb{Z}$ is bilaterally deterministic then $\ch[Z_{[0,k]}][Z_{(-\infty, -\ell]},Z_{[k+1 + m, \infty)}]=0$ for any $k,\ell,m\in\N$ and by taking $\ell,m\to\infty$, we easily obtain $Z_{[-k,k]}\in \ts[d] \text{ for each }k\geq 1$. Cf.~also Remark~\ref{rem11}.
           Informally, ``given the far past and the distant future, the present can be reconstructed''~\cite{Ornstein_1975}.
        \end{rema}
        \begin{rema}\label{rema1.12}
           Given  a stationary finitely-valued process $\mb{Z}$, let
           \[
              \ts[p] := \bigcap_{n\geq1} \sigma\left(Z_{(-\infty, -n]}\right), \qquad \ts[f] := \bigcap_{n\geq1} \sigma \left(Z_{[n, \infty)}\right)
           \]
           denote, respectively, the tail $\sigma$-algebra corresponding to the past and to the future. By a celebrated result of Pinsker~\cite{MR0152628}, $\ts[p] \tos[=][\Pro] \ts[f]\tos[=][\Pro]\Pi$, where $\Pi$ denotes the Pinsker $\sigma$-algebra (i.e., the largest zero entropy sub-$\sigma$-algebra). Thus, the following conditions are equivalent (cf.\ Remark~\ref{rema1.10}):
           \begin{itemize}
              \item
                    $\h[\mb{Z}]=0$,
              \item
                    $Z_{[-k,k]}\in \ts[p]$ for each  $k\geq 1$,
              \item
                    $\sigma(\mb{Z})=\ts[p]$.
           \end{itemize}
        \end{rema}
        A direct consequence of Remark~\ref{rema1.10} and Remark~\ref{rema1.12} is the following observation:
        \begin{coro}\label{kabel}
           Suppose that $\h[\mb{Z}]>0$. Then  $\mb{Z}$ is not bilaterally deterministic whenever $\ts[d]=\ts[p]$. In particular, this happens if $\ts[d]$ is trivial. 
        \end{coro}

        Notice that from this point of view, stationary processes can be split into three pairwise disjoint classes: 
\begin{enumerate}[label={(\alph*)}]
\item\label{CA}
of zero entropy rate (they are automatically bilaterally deterministic),
\item\label{CB}
of positive entropy rate that are bilaterally deterministic,
\item\label{CC}
of positive entropy rate but not bilaterally deterministic.
\end{enumerate}
           Class~\ref{CC} includes the following positive entropy rate processes:
           \begin{itemize}
              \item {exchangeable processes},
              \item Markov chains,
              \item weakly Bernoulli processes (here $\ts[d]$ is trivial),
           \end{itemize}
for more details, see Section~\ref{przyklady}. Theorem~\ref{main} allows us to ``compare'' a large subclass of processes from class~\ref{CA} with processes from class~\ref{CC}, see Corollaries~\ref{WWW} and~\ref{VVV} below. In particular, the zero entropy class that we have in mind contains {\bf all} $\mathscr{B}$-free systems (considered with the Mirsky measure), cf.\ Proposition~\ref{dychotomy} and Corollary~\ref{wolne}. We leave as an open problem to find answers to analogous questions on the relations between class~\ref{CA} with class~\ref{CB}.

Notice that
        \begin{eqsn}
           &\E_{Y_0=1}\ch[X_{[1,r_1-1]}][X_{(-\infty,0]},X_{\{r_1,r_2,\dots\}}]_{|r_i = R_i} \geq \E_{Y_0=1}\ch[X_{[1,r_1-1]}][X_{(-\infty,0]},X_{[r_1,\infty)}]_{|r_1 = R_1}\\
           &=\sum_{k\geq 1}\pof[R_1=k+1][Y_0=1]\ch[X_{[1,k]}][X_{(-\infty,0]},X_{[k+1,\infty)}].
        \end{eqsn}
        Moreover, if $\mb{X}$ fails to be bilaterally deterministic, then, for all $k$ sufficiently large, we have
        \begin{equation}\label{dodatnie}
           \ch[X_{[1,k]}][X_{(-\infty,0]},X_{[k+1,\infty)}]>0.
        \end{equation}
        Thus, using Theorem~\ref{main}\ref{jeszcze_jeden}, we obtain immediately the following:
        \begin{coro}[answer to Question~\ref{Q:procesy}\ref{qa3}]\label{WWW}
        Suppose that $(\mb{X},\mb{Y})$ is {\em good }and $\mb{Y}$ is ergodic of zero entropy rate. If additionally 
\begin{equation}\label{nieskoczenie_wiele}
\pof[R_1=k]>0 \text{ for infinitely many }k\in\N
\end{equation}         
and $\mb{X}$ is not bilaterally deterministic then $\h[\mb{M}]  < \h[\mb{X}]$.
        \end{coro}
\begin{rema}\label{bernu}
In fact, if we know more about $\mb{X}$ than just~\eqref{dodatnie} then the assumption that $\pof[R_1  = k] > 0$ for infinitely many $k\in \N$ can be relaxed and we can still have $\h[\mb{M}]  < \h[\mb{X}]$. E.g.\ if $\mb{X}$ is Bernoulli then we will always have $\h[\mb{M}]<\h[\mb{X}]$ whenever $(\mb{X},\mb{Y})$ is good and $\mb{Y}$ is of zero entropy rate.
\end{rema}

A natural question arises what happens when~\eqref{nieskoczenie_wiele} fails to hold. Suppose that our processes are of dynamical origin and the underlying dynamical system is a transitive symbolic dynamical system. Namely, take $\mb{w}\in \{0,1\}^\Z$ such that the support of $\mb{w}$ is unbounded both from below and from above, and suppose that $\mb{w}$ is quasi-generic along some subsequence for an invariant zero entropy measure $\nu$. Let $Y$ be the orbit closure of $\mb{w}$ under the left shift~$S$ and let $\mb{Y}\sim \nu$ be the corresponding stationary process. Clearly,
\[
\eqref{nieskoczenie_wiele} \implies \text{ the support of $\mb{w}$ contains a two-sided infinite arithmetic progression}.
\]  
It turns out that if we assume that the support of $\mb{w}$ does contain a two-sided (infinite) arithmetic progression then one can obtain a complementary result to~Corollary \ref{WWW}:
\begin{coro}\label{VVV}
Let $\mb{Y}$ be a {\em good}, ergodic process of zero entropy rate. Assume additionally that there exists $L\geq 1$ such that  for a.e.\ realization $\mb{y}$, the corresponding return time sequence $\mb{r}$ contains an arithmetic progression of difference $L$. Then there exists a stationary binary process $\mb{X}$ that is not bilaterally deterministic and such that $\h[\mb{M}]=\h[\mb{X}]$.
\end{coro}

Let us turn now to the interpretation of Corollaries~\ref{WWW} and \ref{VVV} from the point of view of $\mathscr{B}$-free systems. It turns out that in this case we have the following dychotomy:
\begin{prop}\label{dychotomy}
Let $\mathscr{B}\subset \mathbb{N}$ and let $\eta$ be the characteristic function of the corresponding $\mathscr{B}$-free set. Then exactly one of the following holds:
\begin{itemize}
\item $(X_\eta,S)$ is proximal and then for infinitely many $k\geq 1$ the block of the form $10\ldots 01$ (with $k$ zeros between the 1's) is of positive Mirsky measure $\nu_\eta$,
\item $(X_\eta,S)$ is not proximal and then $\eta$ contains a two-sided infinite arithmetic progression.
\end{itemize}
\end{prop}
As a direct consequence of Corollaries~\ref{WWW},~\ref{VVV} and Proposition~\ref{dychotomy}, for $\mathscr{B}$-free systems we have the following result:
\begin{coro}\label{wolne}
Let $\mathscr{B}\subset \mathbb{N}$. Then $(X_\eta,S)$ is proximal if and only if for any $\mb{X}$ that is not bilaterally deterministic, such that $\mb{X}\amalg \mb{Y}$, we have $\h[\mb{M}]<\h[\mb{X}]$.
\end{coro}

Finally, let us remark that $\mb{X}$ in Corollary~\ref{VVV} can be chosen to be very weakly Bernoulli (i.e.\ as a dynamical system, isomorphic to a Bernoulli process~\cite{MR346132}) (compare Example~\ref{i} below and Remark~\ref{bernu}). That is, for $\mb{Y}$ as in Corollary~\ref{VVV}, we can find a measure-theoretic dynamical system $(X,\mathcal{B},\mu,T)$ with two stochastic representations $\mb{X}$ and $\mb{X}'$ (both not bilaterally deterministic!) such that $\h[\mb{X}\cdot \mb{Y}]<\h[\mb{X}]=\h[\mb{X}']=\h[\mb{X}'\cdot \mb{Y}]$. More than that, in some cases $\mb{X}'$ can be retrieved from $\mb{X}'\cdot \mb{Y}$. This matches well with the fact that the notion of a bilaterally deterministic process is not stable under taking various process representation of a given dynamical system~\cite{Ornstein_1975}. It makes the situation completely different from the one in~\cite{Fu-Pe-We}, where the results are purely ergodic-theoretic.

    \subsection{Dictionary between ergodic theory and probability theory}

        In our paper, both ergodic-theoretic and stochastic questions and tools are often intertwined. Let us now give some samples of ergodic-theory results translated into the language of stochastic processes. Our basic object is an ergodic-theoretic dynamical system $(\mc{X}^\Z, \mu, S)$, where $S$ stands, as usual, for the left shift, together with a subset $A \subset X$ satisfying $\mu(A) > 0$. Recall that for $x\in A$, the \textit{first return time} $n_A$ is defined as $n_A(x) = \inf\left\{n\ge 1\;|\; S^n x \in A\right\}$ and the corresponding \textit{induced transformation} as $S_A(x) = S^{n_A(x)}(x)$, with the corresponding conditional measure $\mu_A$ being invariant under $S_A$.
	
        Fix now a stationary process $\mb{X} = \proc$ on $(\Omega,\mathcal{F},\mathbb{P})$, with distribution $\mu$, i.e.\ $\mb{X}\sim \mu$. This is a stochastic counterpart of $(\mc{X}^\Z,\mu,S)$, cf.\ also Section~\ref{se:fur}. Left shift $S$ naturally acts on processes by $S \mb{X} = \left(X_{i + 1}\right)_{i \in \Z}$. In particular,
        \[
           S_\ast(\mu) = \mu \text{ precisely if }S\mb{X} \sim \mb{X}.
        \]
        Similarly, $\mu_A$ corresponds to the distribution of $\mb{X}$ \textbf{under} $\Pro_{\mb{X} \in A}$.  To see how one should interpret $S_A$ in terms of stochastic processes, let $R_{A} = \inf\left\{n \ge 1\;|\; S^n\mb{X} \in A\right\}$ be the first return time, defined on $\mb{X}\in A$, cf.\ \eqref{procesR}. Now, we set $S_A \mb{X} = \left(X_{i + R_A}\right)_{i \in \Z}$ and one can easily check that
        \[
           S_A \mu_A  = \mu_A  \text{ precisely if }S_A \mb{X} \sim \mb{X} \text{ under }\Pro_{\mb{X} \in A}.
        \]
	Finally, recall that $h(\mu)=\h[\mb{X}]$.
        
        Let us present a summary of some classical ergodic theorems (formulated for $(\mc{X}^\Z,\mu,S)$), with their counterparts for random processes. \\
        \noindent\begin{tikzpicture}
           \matrix (first) [table]
           {
           & Ergodic    &  Probabilistic\\
           Ergodicity of $\mu$\footnotemark{}& $\frac1n\sum_{i=0}^{n-1}S^i f \to \int f\, d\mu$ & \small{$\frac1n\sum_{i=0}^{n-1}f(S^{i}\mb{X})  \to \E f(\mb{X})$} \\
           Poincar\'{e} Rec.& $\mu_{A}\left(\{x : S^kx \in A \textnormal{ i.o.}\}\right) = 1$ & $\pof[S^k\mb{X} \in A\textnormal{ i.o.}][\mb{X} \in A] = 1$\\
           Kac's Lemma	&  \small{$\int_A n_A d\mu_A  =1$}     & \small{$\pof[\mb{X} \in A] \E_{\mb{X} \in A} R_A = 1$}\\
           Invariance of $\mu_A$ & \small{$S_A\mu_A  =\mu_{A}$ } & \small{$S_A\mb{X} \sim \mb{X}$, under $\Pro_{\mb{X} \in A}$}\\
           Ergodicity of $\mu_A$ & $\frac1n\sum_{i=0}^{n-1}S^i f \to \int f\, d\mu_A$& \small{$\frac1n\sum_{i=0}^{n-1}f(S_A^{i}\mb{X})  \to \E_{\mb{X}\in A} f(\mb{X})$}\\
           Maker's ET\footnotemark{} & \small{$\frac1n\sum_{i=0}^{n-1}S^{i}f_{n - i}  \to \int f d\mu$} & \small{$\frac1n\sum_{i=0}^{n-1}f_{n - i}(S^{i}\mb{X})  \to \E f(\mb{X})$} \\
           };
        \end{tikzpicture}
        \addtocounter{footnote}{-2}
        \stepcounter{footnote}\footnotetext{Here, in fact, we state Birkhoff ergodic theorem under the assumption that $\mu$ is ergodic.}
        \stepcounter{footnote}\footnotetext{ET stands for ``ergodic theorem''.}
        
       We owe the reader a word of explanation concerning the abbreviations in the table above. The convergence of ergodic averages is always meant a.e.\ / a.s.\ with respect to the appropriate underlying measure ($\mu$ or $\mu_A$ / $\Pro$ or $\Pro_{\mb{X} \in A}$). Also, we tacitly assume that all required assumptions are satisfied, e.g.\ functions appearing in ergodic averages are integrable with respect to the underlying measure. Finally, let us give some details concerning Maker's ergodic theorem~\cite{Maker_1940} which will play a central role in the proof of Theorem~\ref{main}~\ref{mainA}. We recall it now (in the ergodic-theoretic language, i.e.\ as in~\cite{Hochman-notes}, under the extra assumption that $T$ is ergodic).
        \begin{theo}[Maker's ergodic theorem]
           Let $(X,\mu,T)$ be an ergodic measure-theoretic dynamical system. Let $f\in L_1(\mu)$ and $f_n\to f$ $\mu$-a.e. Suppose that $\sup_n |f_n|\in L_1(\mu)$. Then
           \[
              \frac1n\sum_{i=0}^{n-1}T^if_{n-i} \to \E_\mu f \text{ a.e.}
           \]
        \end{theo}

        Let us now return to our general setting, with standing assumptions~\ref{RAZ} and~\ref{DWA} on $\mb{X}$ and~$\mb{Y}$.   Consider the \textit{inter-arrival process} $\mb{T} = \proc[T]$, where
        \begin{equation}\label{definitino of inter-arrival}
           T_i = R_i - R_{i - 1}
        \end{equation}
        and the \textit{return-process} $\mb{R}$ is as in \eqref{procesR}. Thus, $T_i$ tells us how much time elapses between $(i-1)$'th and $i$'th visit of $\mb{Y}$ to the state $1$.
        
        \begin{rema}[Factor of a random process]\label{rema:factor}
        	Recall that whenever $\mb{Y}$ is ergodic, the return process $\mb{R}$ and thus also $\mb{T}$ is well-defined. Moreover, $\mb{T}$ can be regarded as a \textbf{factor} of $\mb{Y}$ in the ergodic-theoretic sense. More precisely, by the very definition of $\mb{T}$, there is a natural measurable function $\pi\colon \left(\{0, 1\}^\Z, S_{[1]}, \mc{L}(\mb{X}|\; \Pro_{Y_0 = 1})\right) \to \left(\Z^\Z, S, \mc{L}(\mb{T}\;|\; \Pro_{Y_0 = 1})\right)$ such that $\pi(\mb{X}) = \mb{T}$ almost surely, where $\mc{L}(\cdot \;|\; \cdot)$ stands for the "distribution of $\cdot$ under $\cdot$", $[1] = \{\mb{y}\; |\; y_0 = 1\}$ and $S_{[1]}$ is the corresponding induced shift operator (cf.\ the beginning of this section). Clearly, $\pi S_{[1]} = S \pi$. In particular, since $\mb{Y} \sim S_{[1]}\mb{Y}$ and $\mb{Y}$ is ergodic (under $\Pro_{Y_0 = 1}$), we get that  $ \mb{T}$ is stationary and  ergodic (under $\Pro_{Y_0 = 1}$) as well. 
        \end{rema}

       As a consequence of the above remark, we can apply Maker's ergodic theorem to $\mb{T}$, which results in the following corollary:
        \begin{coro}
           Suppose that $\sup_{i\in\N} g_i(\mb{T}) \in L_1(\Pro_{Y_0 = 1})$ and $g_i\xrightarrow{\Pro_{Y_0 = 1}\;a.s.} g$. Then, $\Pro_{Y_0 = 1}$ a.s.,
           \begin{equation*}
              \lim_{n\to\infty}\frac1n\sum_{i=0}^{n-1}g_{n - i} \left({(T_{i + j})}_{j\in\Z}\right) \tos[=][] \E_{Y_0 = 1} g\left(\mb{T}\right).
           \end{equation*}           
        \end{coro}

\section{Examples, comments and proofs}
    \subsection{Examples of non-bilaterally deterministic processes}\label{przyklady}

 In the subsections below we tacitly assume that $(\mb{X},\mb{Y})$ is {\em good}.

        \subsubsection{Exchangeable processes}
           \begin{defi}
              We say that a process $\mb{X}$ is \textit{exchangeable} if for any $n\in \N$ and distinct times $i_1, i_2, \ldots, i_n$,
              \begin{equation*}
                 \left(X_{i_1}, X_{i_2}, \ldots, X_{i_n}\right) \sim \left(X_{1}, X_{2}, \ldots, X_{n}\right).
              \end{equation*}
              In other words, the distribution of $\mb{X}$ is invariant under finite permutations.
           \end{defi}
           \begin{rema}
              Let $\mb{X} = \proc$ be {exchangeable}. By a celebrated result of de Finetti~\cite{finetti31} (cf.\ also \cite{MR76206}), this is equivalent to $\mb{X}$ being a convex combination of i.i.d.\ processes. Thus, there exists a random variable $\Theta$ such that, conditionally on $\Theta$, $\mb{X}$ is i.i.d. Note that this ensures that $\ph[X] > 0$, unless
              $X_i = f_i(\Theta)$ for some Borel functions $f_i$. Indeed, 
              \begin{equation*}
                 \h[X_1, \ldots, X_n] \ge \ch[X_1, \ldots, X_n][\Theta] = \sum_{i = 1}^{n}\ch[X_i][\Theta] = n\ch[X_1][\Theta],
              \end{equation*}
              which gives $\ph[X] \ge \ch[X_1][\Theta]$. Therefore,  $\ph[X] = 0$ implies $X_i = f_i(\Theta)$.
           \end{rema}
           \begin{rema}\label{olshen}
              Olshen in~\cite{MR288797} showed that if $\mb{X} = \proc$ is exchangeable then
              \begin{equation*}
                 \mathcal{I} = \mathcal{E} = \ts[d] = \ts[f] = \ts[p],
              \end{equation*}
              (as measure-algebras),
              where $\mc{I}, \mathcal{E}$ denote the $\sigma$-algebra of shift-invariant and finite permutation invariant sets respectively and $\ts[d]$, $\ts[f]$, $\ts[p]$ are double, future, past tails respectively.
           \end{rema}
           As an immediate consequence of Remark~\ref{olshen} and Corollary~\ref{kabel}, we obtain the following:
           \begin{coro}
              Suppose that $\mb{X}$ is exchangeable. Then $\h[\mb{X}]>0$ if and only if $\mb{X}$ is not bilaterally deterministic.
           \end{coro}

           \begin{prop}\label{lowerboundforex}
              Suppose that $\mb{X}$ is exchangeable. Then $\cph[M][Y] =\Pro\left(Y_0 = 1\right) \h[\mathbf{X}]$.
              \begin{proof}
                 It follows from the exchangeability of $\mb{X}$ that for any negative distinct times $ r_{-i}$, $i\in\N$,
                 \begin{equation*}
                    \ch[X_0][X_{\{r_{-1},r_{-2}, \ldots\}}] = \ch[X_0][X_{\{-1, -2, \ldots\}}] =\h[\mb{X}]
                 \end{equation*}
                 It remains to use Theorem~\ref{main}~\ref{mainA}.
              \end{proof}
           \end{prop}

        \subsubsection{Markov chains}\label{LM}
           Recall that a process $\mb{X}$
           is a Markov chain if, for every time $i\in\Z$,
           conditionally on $X_i$, $X_{(-\infty, i - 1]}$ is independent of $X_{[i + 1, \infty)}$. Colloquially, given present, the past and the future are independent. This immediately leads to the following corollary of Theorem~\ref{main}~\ref{mainA}:
           \begin{coro}\label{MA}
              If $\mb{X}$ is a Markov chain (and $(\mb{X},\mb{Y})$ is good) then
              \begin{equation*}
                 \h[\mb{M}]= \pof[Y_0 = 1]\sum_{k = 1}^{\infty}\pof[R_{1}  = k][Y_{0} = 1] \ch[X_k][X_{0}].
              \end{equation*}

           \end{coro}
           \begin{rema}
              Corollary~\ref{MA} easily extends to the case of $k$-Markov chains but for simplicity sake we decided to present it for $k = 1$.
           \end{rema}
           \begin{rema}
              Let $\mb{X} = \proc[X][\N]$ be a finitely-valued Markov chain, $X_i \in \mc{X}$. It is well-known (see \cite{feller2008introduction}, Chapter XV, Section $6$, Theorem $3$, page $392$) that we can uniquely decompose the state space $\mc{X}$ into disjoint union
              \begin{equation}\label{Markov decomposition}
                 \mc{X} = C \sqcup D_1 \sqcup D_2 \sqcup \cdots \sqcup D_k,
              \end{equation}
              where $C$ is the set of transient states and $D_i$ are closed sets. If $\mb{X}$ starts  in $D_j$ (i.e.\ $X_0 \in D_j$) then it remains in $D_j$ forever. If $X_0 \in C$ then $\mb{X}$ stays in $C$ for finite time and jumps to some $D_j$ (and never leaves $D_j$ afterwards). Moreover (see \cite{feller2008introduction}, Chapter XV, Section $7$, Criterion, page $395$), if $\pi$ is a stationary measure then necessarily $\pi(C) = 0$.
           \end{rema}

           Now suppose that a {bilateral}, finitely-valued Markov chain $\mb{X} =\proc$ is stationary (thus, $C = \emptyset$ in \eqref{Markov decomposition}). Fix $1\leq j\leq k$ and let $\mb{X}_{D_j}$ stand for $\mb{X}$ conditioned on $X_0 \in D_j$. By the definition of $D_j$, process $\mb{X}_{D_j}$ is an irreducible (equivalently, ergodic), stationary Markov chain. Now, let $p_j$ be the period of $\mb{X}_{D_j}$. Then $D_j$ can be decomposed into $p_j$ disjoint sets (see \cite{chung1967markov}, Chapter $1$, Section $3$, Theorem $4$)
           \begin{equation*}\label{Markov periodic decomposition}
              D_j = D_{j, 0} \sqcup \cdots \sqcup  D_{j, p_j - 1}
           \end{equation*}
           such that $\pof[X_1 \in D_{j, (\ell + 1)\bmod p_j}\;|\; X_0 \in D_{j, \ell}] = 1$. Using Corollary $2$ from \cite{Markovtail},  we get that
           \begin{equation*}
              \ts[d]\left(\mb{X}_{D_j}\right) =  \ts[p]\left(\mb{X}_{D_j}\right)= \ts[f]\left(\mb{X}_{D_j}\right)=  \sigma\left\{\left\{X_0 \in D_{j, 0}\right\}, \left\{X_0 \in D_{j, 1}\right\}, \ldots, \left\{X_0 \in D_{j, p_j - 1}\right\}\right\}.
           \end{equation*}
           Note that Corollary $2$ from \cite{Markovtail} is stated only for $\ts[f]$ but a perusal of the proofs of Theorem~$1$ and Corollaries~$1$ and $2$ therein gives the same result for $\ts[d]$. Thus, $\mb{X}$, conditionally on $X_0 \in D_{j, l}$, has trivial tail $\sigma$-algebras. This immediately leads to
           \begin{equation}\label{tails of Markov}
              \ts[d]\left(\mb{X}\right) = \ts[p]\left(\mb{X}\right) = \ts[f]\left(\mb{X}\right) = \sigma\left\{\left\{X_0 \in D_{j, \ell}\right\}\;|\;1\le j \le k, 0 \le \ell \le p_j\right\}.
           \end{equation}
           Indeed, if for example $A \in 	\ts[d]\left(\mb{X}\right)$ then, for all $j, \ell$, $\pof[A\;|\;X_0 \in D_{j, \ell}] \in \{0, 1\}$ which yields \eqref{tails of Markov}.
           As a consequence of \eqref{tails of Markov}, we obtain the following:
           \begin{coro}
              Suppose that $\mb{X}$ is a stationary finitely-valued Markov-chain. Then $\h[\mb{X}]>0$ if and only if $\mb{X}$ is not bilaterally deterministic.
           \end{coro}

           \begin{rema}
              Since $\h[\mb{X}]=\ch[X_1][X_0]=\ch[X_{i + 1}][X_i]$, it follows that $\h[\mb{X}]=0$ if and only if, for every $i\in\Z$, $X_i = f_i(X_{0})$ for some functions $f_i$. It is not hard to see that if for every $x\in\mc{X}$, $\pof[X_0 = x] >0$, then every $f_i$ must be a bijection on $\mc{X}$. Moreover, by the stationarity of $\mb{X}$, for $f_1(x) = y$,  we get
              \begin{eqsn}
                 \pof[X_0 = x] &= \pof[X_0 = x, f_1(x) = y] = \pof[X_0 = x, f_1(X_0) = y] \\
                 &= \pof[f_i(X_0) = x, f_{i + 1}(X_0) = y] = \pof[X_0 = f_i^{-1}(x)]\iof[f_{i + 1}\left(f_i^{-1}(x)\right) = f_1(x)].
              \end{eqsn}
              Thus, necessarily, $f_{i + 1}\left(z\right) = f_1(f_i(z))$. Consequently, if we set $f:=f_1$ then $f_{i} = f^{\circ i}$. Moreover, $f$ must be such that, for all $x$, $\pof[X_0 = x] = \pof[X_0 = f(x)]$.

              Therefore, if $\mb{X}$ is bilateral, finitely-valued, stationary Markov chain, with $\pof[X_0 = x] >0$ for all $x \in \mc{X}$, then the following are equivalent:
              \begin{itemize}
                 \item $\mb{X}$ is bilaterally deterministic;
                 \item there exist a bijection $f\colon\mc{X} \to \mc{X}$, such that $X_i = f^{\circ i}(X_0)$ and for all $x\in\mc{X}$, $\pof[X_0 = x] = \pof[X_0 = f(x)]$.
              \end{itemize}
           \end{rema}

        \subsubsection{Weakly Bernoulli processes}
           Weakly Bernoulli processes were introduced by Friedman and Ornstein~\cite{MR274718} and belong to the classics of ergodic theory. Equivalently, one speaks of finitely determined processes. Recall that any process $\mb{X}$ that is weakly Bernoulli is also very weakly Bernoulli (i.e.\ as a dynamical system, it is isomorphic to a Bernoulli process~\cite{MR346132}). In particular, $\h[\mb{X}]>0$. We refer the reader, e.g., to~\cite{MR0442198} for more information on the subject.

           Suppose now that $\mb{X}$ is weakly Bernoulli. Then $\ts[d]$ is trivial (see, e.g., Proposition 5.17 in \cite{MR2325294}). Therefore, as an immediate consequence of Corollary~\ref{kabel}, we obtain the following:
           \begin{coro}
              Suppose that $\mb{X}$ is weakly Bernoulli. Then $\mb{X}$ is not bilaterally deterministic.
           \end{coro}
           In fact, the results in \cite{MR2325294} are formulated in a different language. One more notion, equivalent to the weak Bernoulli property, is absolute regularity. It first appeared in works of Volkonskii and Rozanov~\cite{MR121856,MR0137141} who, in turn, attribute it to Kolmogorov. Fix a probability space $(\Omega,\mathcal{F},\mathbb{P})$. Let $\mathcal{A},\mathcal{B}\subset \mathcal{F}$ be sub-$\sigma$-algebras and let
           \[
              \beta(\mathcal{A},\mathcal{B}):=\sup\frac12\sum_{i=1}^{I}\sum_{j=1}^{J}|\mathbb{P}(A_i\cap B_j) - \mathbb{P}(A_i)\mathbb{P}(B_j)|,
           \]
           where the supremum is taken over all (finite) partitions $\{A_1,\dots, A_I\}$, $\{B_1,\dots, B_J\}$ of $\Omega$, with $A_i\in \mathcal{A}$, $B_j\in \mathcal{B}$ for each $i,j$.
           Now, given a process $\mb{X}$, for $-\infty\leq J < L \leq \infty$, we define the $\sigma$-algebra
           \[
              \mathcal{F}_J^L:=\sigma(X_k : J\leq k\leq L).
           \]
           Then, for each $n\geq 1$, we define the following $\beta$-dependence coefficients:
           \[
              \beta(n):=\sup_{j\in\Z}\beta(\mathcal{F}_{-\infty}^{j},\mathcal{F}_{j+n}^{\infty}).
           \]
           We say that $\mb{X}$ is \textit{absolutely regular} (or $\beta$-mixing) if $\beta(n)\to 0$ as $n\to\infty$.

           Berbee, in~\cite{MR876396}, studied $\beta$-dependence coefficients for stationary ergodic processes. He showed that
           \[
              \lim_{n\to\infty} \beta(n)= \beta = 1-\frac1p\text{ for some }p\in \N\cup \{\infty\}.
           \]
           Moreover, he proved that if $\beta<1$ then $\ts[d]=\ts[p]$. As a consequence of his result and of Corollary~\ref{kabel}, we have:
           \begin{coro}
              Suppose that $\mb{X}$ is a stationary ergodic process with $\beta<1$. Then $\mb{X}$ is not bilaterally deterministic.
           \end{coro}

    \subsection{Proof of the main technical result (Theorem~\ref{main})}
    \subsubsection{Part~\ref{mainA}}
           By the chain rule (cf.\ \eqref{chr}), we have
           \begin{equation}\label{chainrule}
              \ch[M_{[0, n]}][Y_{[0, n]}] = \sum_{k = 0}^{n} \ch[M_k][Y_{[0, n]}, M_{[0, k)}] =: \sum_{k = 0}^{n} H_{k, n}.
           \end{equation}
           Fix $0 \le k \le n$. Since $M_k = X_k \cdot Y_k$ and $\mb{X} \amalg \mb{Y}$, we easily get that conditionally on $\left(Y_{[0, k]}, M_{[0, k)}\right)$, $M_k$ is independent of $Y_{[k + 1, n]}$. In other words,
           \begin{equation*}
              H_{k, n} = H_k =  \ch[M_k][Y_{[0, k]}, M_{[0, k)}].
           \end{equation*}
           Now, using the definition of Shannon conditional entropy, the fact that on the event $Y_k=0$,  we have $M_k\equiv 0$, whereas on $Y_k = 1$, we have $M_k = X_k$ and the stationarity of the $(\mb{X}, \mb{Y})$, we get
           \begin{equation*}
              H_k = \pof[Y_k = 1]\ch[X_k][Y_{[0,k)},M_{[0,k)}][Y_k=1] = \pof[Y_0 = 1]\ch[X_0][Y_{[-k,0)},M_{[-k,0)}][Y_0=1].
           \end{equation*}
           Moreover, if $Y = Y_{[-k,0)}$, $M = M_{[-k,0)}$, $y = y_{[-k,0)}$, $m  = m_{[-k,0)}$, $s_{-k} = \sum_{i = -k}^{-1}y_i$, $r_{-s_{-k}} < \cdots < r_{-1}$ are such that $y_{r_{-i}} = 1$, then
           \begin{equation*}
              \pof[Y = y , M = m][Y_0 = 1] =
              \begin{dcases}
                 \pof[Y = y][Y_0 = 1]\pof[X_{\{r_{-1}, \ldots, r_{s_{-k}}\}} = m_{\{r_{-1}, \ldots, r_{s_{-k}}\}}], \;\; & s_{-k} > 0, \\
                 \pof[Y = y][Y_0 = 1], \;\;                                                                              & s_{-k} = 0, \\
              \end{dcases}
           \end{equation*}
           whenever $m\leq y$ coordinatewise (otherwise, we get zero). This implies that
           \begin{eqsn}
              H_k &= \pof[Y_0 = 1]\pof[S_{-k} = 0][Y_0 = 1]\h[X_0] \\
              &\quad +\pof[Y_0 = 1]\E_{Y_0 = 1} \iof[{S_{-k} > 0}]\ch[X_0][X_{\{r_{-1}, \ldots, r_{s_{-k}}\}}]_{|_{r_{-i} = R_{-i}, s_{-k} = S_{-k}}}.
           \end{eqsn}

           Since $\mb{Y}$ visits $1$ a.s.\ infinitely many times (in the past),
           \[
              \pof[S_{-k} = 0][Y_ 0 = 1] \to 0 \text{ as }k\to\infty.
           \]
           Moreover, $\Pro_{Y_0 = 1}$ a.s., we have $\iof[{S_{-k} > 0}] \to 1$ and
           \begin{equation*}
              \ch[X_0][X_{\{r_{-1}, \ldots, r_{s_{-k}}\}}]_{|_{r_{-i} = R_{-i}, s_{-k} = S_{-k}}} \to \ch[X_0][X_{\{r_{-1}, r_{-2}, \ldots, \}}]_{|_{r_{-i} = R_{-i} }}.
           \end{equation*}
           Thus, by the bounded convergence theorem, we get that
           \begin{equation*}
              H_k \to \pof[Y_0 = 1]\E_{Y_0 = 1}\ch[X_0][X_{\{r_{-1}, r_{-2}, \ldots, \}}]_{|_{r_{-i} = R_{-i} }},
           \end{equation*}
           which, by \eqref{chainrule}, concludes the proof of Theorem~\ref{main}~\ref{mainA}.

        \subsubsection{Part \ref{jeszcze_jeden}}
           First, we will prove a technical lemma.

           \begin{lema}\label{auxiliary lemma1}
              We have
              \begin{equation*}
                 \ch[\mb{X}\cdot\mb{Y}][\mb{Y}]=\lim\limits_{n \to \infty}\frac{1}{n}\E\iof[S_n > 0]\h[X_{r_0}, X_{r_1},\ldots, X_{r_{s_n - 1}}]_{|_{r_i = R_i,  s_n = S_n}}.
              \end{equation*}
              \begin{proof}
                 Since for any $k\in\Z$, on the event $Y_k=0$,  we have $M_k\equiv 0$, it follows that
                 \begin{equation*}
                    \ch[M_{[0, n]}][Y_{[0, n]}] = \pof[S_n > 0]\sum_{y_{[0,n]}} \pof[Y_{[0, n]} = y_{[0, n]}][S_n > 0]\h[M_{[0, n ]}][Y_{[0, n]} = y_{[0, n]}].
                 \end{equation*}
                 Moreover, if $s_n = \sum_{i = 0}^n y_i>0$ then
                 \[
                    \pof[M_{[0,n]}=m_{[0,n]}][Y_{[0,n]}=y_{[0,n]}]=\pof[X_{r_0}=m_{r_0},\dots, X_{r_{s_n - 1}}=m_{r_{s_n - 1}}],
                 \]
                 whenever $m_{[0,n]}\leq y_{[0,n]}$ coordinatewise (otherwise, we get zero).
                 Hence,
                 \begin{equation*}
                    \h[M_{[0, n ]}][Y_{[0, n]} = y_{[0, n]}] = \h[X_{r_0}=m_{r_0},\dots, X_{r_{s_n - 1}}=m_{r_{s_n - 1}}],
                 \end{equation*}
                 which results in
                 \begin{align*}
                    \ch[M_{[0, n]}][Y_{[0, n]}] & = \pof[S_n > 0]\E_{S_n > 0}\h[X_{r_0},\dots, X_{r_{s_n - 1}}]_{|_{r_i = R_i,  s_n = S_n}}
                    \\ &= \E\raz_{S_n > 0}\h[X_{r_0},\dots, X_{r_{s_n - 1}}]_{|_{r_i = R_i,  s_n = S_n}}.
                 \end{align*}
                 This completes the proof.
              \end{proof}
           \end{lema}

           Notice now that
           \begin{equation*}
              \frac{1}{n}\h[X_{r_0},\dots,X_{r_{s_n-1}}]=\frac{1}{n}\h[X_{[0,n]}]- \frac{1}{n}\ch[X_{[0,n]\setminus \{r_0,\dots, r_{s_n-1}\}}][X_{r_0},\dots,X_{r_{s_n-1}}],
           \end{equation*}
           $\lim_{n\to\infty}\frac1n\h[X_{[0,n]}]=\h[\mb{X}]$ and that (by the ergodicity of $\mb{Y}$) we have $\iof[S_n > 0 ] \to 1$. Thus, in order to conclude the proof it remains to find $\lim\limits_{n \to \infty}\E\iof[S_n > 0] H(n, \mb{R})$ where
           \begin{equation*}
              H(n, \mb{r}) := \frac{1}{n}\ch[X_{[0,n]\setminus \{r_0,\dots, r_{s_n-1}\}}][X_{r_0},\dots,X_{r_{s_n-1}}], \quad \mb{r} = \proc[r][\Z].
           \end{equation*}
More precisely, if we show that 
           \begin{equation}\label{abc1}
              \lim\limits_{n \to \infty} H(n, \mb{R}) = \Pro(A_0)\E_{A_0} \ch[X_{[r_0+1, r_{1} - 1]}][X_{(-\infty, r_0]}, X_{\{r_{ 1}, r_{2}, \ldots\}}]|_{r_i=R_i}
           \end{equation}
holds a.e.\ then by the bounded convergence theorem (as $H(n,\mb{R})\leq \h[X_0])$ we will have
\[
\lim_{n\to\infty}\E \iof[S_n>0]H(n,\mb{R}) = \Pro(A_0)\E_{A_0} \ch[X_{[r_0+1, r_{1} - 1]}][X_{(-\infty, r_0]}, X_{\{r_{ 1}, r_{2}, \ldots\}}]|_{r_i=R_i}
\]           
since $\lim_{n\to\infty}\iof[S_n>0]=1$ a.e.\ by the ergodicity of $\mb{Y}$. 

Let
           \[
			A_i=[Y_0=\ldots= Y_{i-1}=0,Y_i=1] \text{ for }i\geq 0
			\]            
(in particular, $A_0=[Y_0=1]$).			

              Fix $\mb{y}$ and $n\in\N$.
              By the chain rule, we get
              \begin{align*}
                  nH(n, \mb{r}) &= \underbrace{\ch[X_{[0,r_0-1]}][X_{\{r_0,\dots, r_{s_n-1}\}}]}_{\Sigma_1(n)} +\underbrace{\ch[(X_{[r_{s_n-1}+1,n]}][X_{r_{s_n-1}}]}_{\Sigma_3(n)} \\
                  & \quad +\underbrace{\sum_{i=0}^{s_n-2}\ch[X_{[r_i+1,r_{i+1}-1]}][X_{[0,r_i]},X_{\{r_{i+1},\dots,r_{s_n-1}\}}]}_{\Sigma_2(s_n-1)} .
              \end{align*}
              We will deal first with the summands $\Sigma_1(n)$ and $\Sigma_3(n)$. Clearly,
              \begin{equation}\label{R1}
                 \frac{1}{n}\Sigma_1(n) \le \frac{1}{n}\h[X_{[0,r_0-1]}] \leq\frac{r_0}{n}H(X_0)\to 0
              \end{equation}
              when $n\to\infty$. Since $s_n=s_{r_{s_n-1}}$, $\frac{s_n}{n} \to \pof[Y_0 = 1] > 0$  (by the ergodicity of $\mb{Y}$) and $r_{s_n - 1} \to \infty$, it follows that
              \begin{equation}\label{R3}
                 \frac{\Sigma_3(n)}{n} \le \frac{n-r_{s_n-1}}{n}H(X_0)=\left(1-\frac{r_{s_n-1}}{s_{r_{s_n-1}}}\cdot\frac{s_n}{n}\right)H(X_0) \to 0.
              \end{equation}
              In order to deal with $\Sigma_2(s_n-1)$, notice that
              \begin{equation}\label{laczymy}
                 \frac{1}{n}\Sigma_2(s_n-1) =\frac{s_n}{n}\frac{1}{s_n}\Sigma_2(s_n-1).
              \end{equation}
              Because of $\frac{s_n}{n} \to \pof[Y_0 = 1]$, it suffices to show that $\mathbb{P}_{A_0}$-a.e.\ we have
              \begin{equation}\label{aim}
                 \lim\limits_{n\to\infty}\frac{1}{n}\Sigma_2(n) = \E_{A_0}\ch[X_{[r_0+1, r_{1} - 1]}][X_{(-\infty, r_0]}, X_{\{r_{ 1}, r_{2}, \ldots\}}].
              \end{equation}
              Using the stationarity of $\mb{X}$, for $t_i = r_i - r_{i - 1}$, we obtain
              \begin{align*}
                  \Sigma_2(n)&=\sum_{i=0}^{n-1}\ch[X_{[r_i+1,r_{i+1}-1]}][X_{[0,r_i]},X_{\{r_{i+1},\dots,r_{n}\}}]                     \\
                  & =\sum_{i = 0}^{n-1}\ch[X_{[1, t_{i + 1} - 1]}][X_{[-r_i, 0]}, X_{\{t_{i + 1}, \dots, t_{i + 1} + \dots + t_{n}\}}].
              \end{align*}
              We would like to apply Maker's ergodic theorem to study the above sum. However, we cannot do it directly due to the term $X_{[-r_i, 0]}$ appearing in the conditional entropies. This obstacle will be overcome by estimating each summand from below and above.

              Fix $k\in\N$.
              Then for every $i$ such that $r_i \ge k$ and for every $j \in \N$, we have
              \begin{multline}\label{thebounds}
                 H_{\infty, j}\left(t_{i + 1}, t_{i + 2}, \ldots \right) \le \ch[X_{[1, t_{i + 1} - 1]}][X_{[-r_i, 0]}, X_{\{t_{i + 1}, \ldots, t_{i + 1} + \cdots t_{i + j}\}}] \\\le H_{k, j}\left(t_{i + 1}, t_{i + 2}, \ldots \right),
              \end{multline}
              where $H_{k, j}\left(t_{i + 1}, t_{i + 2}, \ldots \right) = \ch[X_{[1, t_{i + 1} - 1]}][X_{(-k, 0]}, X_{\{t_{i + 1}, \ldots, t_{i + 1} + \cdots t_{i + j}\}}]$ for $k \in \Z\cup\{\infty\}$. Clearly,
              \begin{align*}
                 H_{k, j}\left(t_{1}, t_{2}, \ldots \right) \xrightarrow{j\to\infty} H_{k}\left(t_{1}, t_{2}, \ldots \right) &:= \ch[X_{[1, t_{1} - 1]}][X_{(-k, 0]}, X_{\{t_{ 1}, t_{1} +  t_{2}, \ldots\}}]\\
                 &=\ch[X_{[r_0+1, r_{1} - 1]}][X_{(-k, r_0]}, X_{\{r_{1}, r_2, \ldots\}}].
              \end{align*}
              By the entropy chain rule and Kac's lemma,
              \begin{equation}\label{majorization}
                 \sup_{k, j \in \N} H_{k, j}(T_{[1, \infty)}) \le \h[X_0] T_1 \in L_1(\Pro_{A_0}).
              \end{equation}
              Therefore, Maker's ergodic theorem implies that, for every $k\in \N\cup\{\infty\}$, $\mathbb{P}_{A_0}$ a.s., we have
              \begin{equation}\label{zmakera}
                 \lim\limits_{n\to\infty}\frac{1}{n}\sum_{i = 0}^{n-1} H_{k, n - i}\left(t_{i + 1}, t_{i + 2}, \ldots \right) \to \E_{A_0} H_{k}\left(T_{1}, T_{2}, \ldots \right).
              \end{equation}
              Using~\eqref{thebounds}, it follows from the definition of $\Sigma_2$ (and the chain rule) that
              \begin{align}
                 \begin{split}\label{szacuj}
                    \frac1n\sum_{i=0}^{n-1}H_{\infty,n-i}(t_{i+1},t_{i+2},\dots)&\leq \frac1n\Sigma_2(n)\\
                    &\leq\frac{t_1+\dots+t_k}{n}H(X_0)+ \frac1n\sum_{i=k}^{n-1}H_{k,n-i}(t_{i+1},t_{i+2},\dots)\\
                    &\leq \frac{t_1+\dots+t_k}{n}H(X_0)+\frac1n\sum_{i=0}^{n-1}H_{k,n-i}(t_{i+1},t_{i+2},\dots),
                 \end{split}
              \end{align}
              with $\frac{t_1+\dots+t_k}{n}H(X_0)\xrightarrow{n \to \infty} 0$. Thus, due to \eqref{zmakera},
              \begin{equation*}
                 \E_{Y_0 = 1} H_{\infty}\left(T_{1}, T_{2}, \ldots \right) \le \lim\limits_{n \to \infty}\frac1n\Sigma_2(n) \le \E_{Y_0 = 1} H_{k}\left(T_{1}, T_{2}, \ldots \right).
              \end{equation*}
              Notice that $H_k \to H_\infty$ as $k\to\infty$. Hence, combining \eqref{majorization} and  the bounded convergence theorem, we obtain
              \begin{equation}\label{zmake}
                 \lim\limits_{n\to\infty}\frac{1}{n} \Sigma_2(n)= \E_{A_0} H_{\infty}\left(T_{1}, T_{2}, \ldots \right)
              \end{equation}
              $\mathbb{P}_{A_0}$ a.s.\ which is exactly \eqref{aim} under $\mathbb{P}_{A_0}$.
              
              It remains to show \eqref{aim} under $\mathbb{P}_{A_i}$ for $i\geq 1$. However, it is a direct consequence of the above and the following lemma:
              \begin{lema}
              Suppose that we have a sequence of measurable functions $(f_n)_{n\geq 1}$ depending on $(T_n)_{n\geq 1}$ and a measurable function $f$ depending on $\mb{Y}$ such that 
              \begin{equation}\label{ta}
              f_n((T_n)_{n\geq 1}) \to f(\mb{Y})
              \end{equation}
              $\mathbb{P}_{A_0}$-a.e. Then~\eqref{ta} holds also $\mathbb{P}_{A_i}$-a.e.\ for each $i\geq 1$.
              \end{lema}
              \begin{proof}
For the sake of simplicity, we assume that $\mb{Y}$ is a cannonical process. Let $B_0\subset A_0$ be the set where~\eqref{ta} holds. We claim that $B_i:=A_i\cap S^{-i}B_0$ is of full measure $\mathbb{P}_{A_i}$ and~\eqref{ta} holds on $B_i$. Indeed, since $S^iA_i\subset A_0$, we have
              \[
              \mathbb{P}_{A_i}(A_i\setminus B_i)=\frac{1}{\mathbb{P}(A_i)}\mathbb{P}(A_i\setminus S^{-i}B_0)=\frac{1}{\mathbb{P}(A_i)}\mathbb{P}(S^iA_i\setminus B_0)\leq\frac{1}{\mathbb{P}(A_i)}\mathbb{P}(A_0\setminus B_0)=0.
              \]
              Moreover, if $\mb{y}\in B_i$ then $S^i\mb{y}\in S^iA_i\cap B_0\subset A_0\cap B_0=B_0$. Since $\mb{y}\in A_i$, it follows immediately that $T_n(\mb{y})=T_n(S^i\mb{y})$ for all $n\geq 1$ which completes the proof.
              \end{proof}

\subsection{General setting: proof of Corollary~\ref{VVV} and related examples}\label{here_are_proofs}
In this section we will study a certain class of {\em good} $(\mb{X},\mb{Y})$ with no entropy drop. We begin by the proof of Corollary~\ref{VVV}.

\begin{proof}[Proof of Corollary~\ref{VVV}]
Let $L\geq 1$ be such that $\textrm{supp}\ \mb{y} \supset L\mathbb{Z} +a$ for some $a$ and for a.e.\ realization $\mb{y}$ of $\mb{Y}$. Let $(X,\mathcal{B},\mu,T)$ be a measure-theoretic dynamical system with entropy less than $\frac{1}{L}\log 2$ and take a measurable partition $X=J \cup J^c$ that is generating for the map $T^L$. Let $Y$ be the orbit closure of $\mb{y}$ in $\{0,1\}^\mathbb{Z}$ under the left shift.

Process $\mb{M}$ corresponds to coding of points in $(X\times Y,T\times S)$ with respect to $J\times C$ (with $C=[1]\subset Y$) and its complement. Using Theorem~\ref{main} \ref{jeszcze_jeden}, we obtain
\begin{equation*}
\h[\mb{M}] =\h[\mb{X}]-\mathbb{P}(A_0)\E_{A_0}\ch[X_{[r_0+1,r_1-1]}][X_{(-\infty, r_0]},X_{\{r_1,r_2,\dots\}}]|_{r_i=R_i}=\h[\mb{X}].
\end{equation*}
(a.e.\ $\mb{r}$ contains a two-sided infinite arithmetic progression with difference $L$, the partition $\{J,J^c\}$ is generic for $T^L$ and thus the conditional entropy in the above formula is equal to zero).
\end{proof}
It would be interesting to know if in the above example $\mb{X}$ can be recovered from $\mb{M}$. Let us see now that this can be the case when $\mb{Y}$ arises from the rotation on two points $\{0,1\}$. We will look at it both from the probabilitic and ergodic-theoretic perspective.

        \begin{exam}\label{przyklad}
           Let $\left(\xi_i\right)_{i\in \Z}$ be a sequence of i.i.d. random variables such that
           \begin{equation*}
              \pof[\xi_0 = 0] = \pof[\xi_0 = 1] = \frac{1}{2},
           \end{equation*}
           an arbitrary (relabelling) $1$-$1$ function $F\colon\{0,1\}^2 \to \{0, 1, 2, 3\}$ and put
           \begin{equation*}
              X_i = F(\xi_i, \xi_{i + 1}), \qquad \mb{Y} \sim \frac{1}{2}(\delta_a+\delta_{Sa}),
           \end{equation*}
           where $a:=(\ldots 0,1,0,1,\underaccent{\circ}{0},1,0,1\ldots)$, $S$ stands for the left shift and $\mb{X} \amalg \mb{Y}$. Since $\mb{X}$ is a Markov chain and $F$ is $1$-$1$, we have
           \begin{equation*}
              \h[\mb{X}] = \ch[X_1][X_0] = \ch[\xi_1, \xi_2][\xi_0, \xi_1] = \ch[\xi_2][\xi_0, \xi_1] = \h[\xi_2]=\log 2.
           \end{equation*}
           Moreover, $\mathbb{P}_{Y_0=1}(R_{-1}=-2)=1$ and therefore
           \[
              \E_{Y_0=1} \ch[X_0][X_{\{r_{-1}, r_{-2}, \cdots\}}]|_{r_i=R_i}=\ch[X_0][X_{-2}]=\h[X_0]=2\log 2.
           \]
           Clearly, for every $j\in \Z$, $\left(X_i\right)_{i \le j} \amalg \left(X_i\right)_{i \ge j + 2}$ yielding
           \begin{equation*}
              \frac{1}{n}f(y_{[0,n]}) = \frac{1}{n}\h[X_{r_1}, \ldots, X_{r_m}] =  \frac{m}{n} \h[X_0] \to \frac{1}{2}\h[X_0].
           \end{equation*}
           Thus, by Theorem~\ref{main}~\ref{mainA}, $\h[\mb{M}] = \frac{1}{2}\h[X_0] = \frac{1}{2} 2 \log 2 =\log 2 = \h[\mb{X}]$. In fact, notice that since $F$ is 1-1, knowing all even (resp.\ all odd) coordinates of a realization $\mb{x}$ of $\mb{X}$ determines its \textbf{all} coordinates. In other words, $\mb{M}$ contains full information about $\mb{X}$.
        \end{exam}
We will now see how to use ergodic-theoretic approach to modify the above idea so that $X_i \in \{0,1\}$ and keep the property $\h[\mb{M}]=\h[\mb{X}]$ and the ability to recover $\mb{X}$ from $\mb{M}$.
           
           \begin{exam}\label{i}
              Let $(X,\mathcal{B},\mu,T)$ be an ergodic automorphism, with $h(\mu)\in (0,\log 2)$ and let $S$ be the rotation on $Y=\{0,1\}$, with the unique invariant measure denoted by $\nu$. Let $\{J,J^c\}$ be a (measurable) generating partition of $X$ for $T$ (the existence of such a partition follows by Krieger's finite generator theorem~\cite{MR259068}) and let $C:=\{1\}\subset Y$ We consider the following stationary process:
              \[
                 \mb{X}=\br*{\raz_J \circ T^i}_{i\in \Z} \text{ and }\mb{Y}=\br*{\raz_C \circ S^i}_{i\in\Z}.
              \]
Then $\mb{M}:=\mb{X}\cdot\mb{Y}$ corresponds to coding of points in the dynamical system $(X\times Y,T\times S)$ with respect to $J\times C$ and its complement:
              \begin{align*}
                  & (x,1)\mapsto (\dots,\underaccent{\circ}{\raz_J(x)},0,\raz_J(T^2x),0,\dots);    \\
                  & (x,0)\mapsto (\dots,\underaccent{\circ}{0},\raz_J(Tx),0,\raz_J(T^3x),\dots).
              \end{align*}
              Equivalently, $\mb{M}$ corresponds to the dynamical system that is a tower of height two above the factor of $T^2$ corresponding to the partition $\{J,J^c\}$.

              Assume now additionally that $h(T)<\frac12\log2$ and the partition $\{J,J^c\}$ is generating for $T^2$ (e.g.\ $T$ can be a Bernoulli automorphism, with entropy less than $\frac12\log2$). Then $\mb{M}$ corresponds to a tower of height two above $T^2$, denoted by $R$, and given by
              \[
              R(x,0)=(x,1),\ R(x,1)=(T^2x,0).
              \]
	      Notice that $R$ is isomorphic to $T\times S$ via the map $\Phi$ given by
	      \[
	      \Phi(x,0)=(x,0),\ \Phi(x,1)=(Tx,1)
	      \]
	      (we easily check that $\Phi\circ R=(T\times S)\circ \Phi$). It follows that
              \begin{equation}\label{brakspadku}
                 \h[\mb{M}]=h(\mu\otimes \nu)=h(\mu)=\h[\mb{X}]>0.
              \end{equation}
In fact, since $\Phi$ is an isomorphism, one can filter out $\mb{X}$ from $\mb{M}$.
\end{exam}

\subsection{$\mathscr{B}$-free systems: proof of Proposition~\ref{dychotomy}}
Let $\mathscr{B}\subset \mathbb{N}$, let $\eta=\iof[\mathcal{F}_\mathscr{B}]$ and let $(X_\eta,S)$ be the corresponding $\mathscr{B}$-free system, with the underlying Mirsky measure $\nu_\eta$.
 Recall that:
\[
h(\widetilde{X}_\eta,S)=\overline{d}(\mathcal{F}_\mathscr{B})=\nu_\eta(1),
\]
so $\nu_\eta\neq \delta_{(\dots,0,0,0,\dots)}$ is equivalent to $h(\widetilde{X}_\eta,S)>0$. Thus, $\nu_\eta\neq \delta_{(\dots,0,0,0,\dots)}$ is necessary and sufficient for the existence of $\kappa$ with $h(\nu_\eta\ast \kappa)>0$. 

\begin{proof}[Proof of Proposition~\ref{dychotomy}]
It was shown in Theorem 3.7 in~\cite{Dy-Ka-Ku-Le} that the following are equivalent:
\begin{itemize}
\item
$(X_\eta,S)$ is proximal,
\item 
$\mathscr{B}$ contains an infinite pairwise coprime subset,
\item
the support of $\eta$ does not contain a two-sided infinite arithmetic progression.
\end{itemize}
Thus, in order to complete the proof of Proposition~\ref{dychotomy}, we need to show that in the proximal case, for infinitely many $k\geq 1$ the block of the form $10\ldots 01$ (with $k$ zeros between the 1's) is of positive Mirsky measure $\nu_\eta$. An important notion in the theory of $\mathscr{B}$-free systems is that of \textit{tautness}~\cite{Ha},defined in terms of the logarithmic density of sets of multiples. We say that $\mathscr{B}$ is taut if for any $b\in\mathscr{B}$, we have
\[
\boldsymbol{\delta}(\mathcal{M}_{\mathscr{B}}) > \boldsymbol{\delta}(\mathcal{M}_{\mathscr{B}\setminus \{b\}}),
\]
where $\boldsymbol{\delta}(A)=\lim_{N\to\infty}\frac{1}{\log N}\sum_{n\leq N}\frac{1}{n}\mathbf{1}_{A}(n)$ for any $A\subset \mathbb{Z}$. It was proved in~\cite{Dy-Ka-Ku-Le} (see Theorem C and Lemma 4.11 therein) that given any $\mathscr{B}$, there exists a taut set $\mathscr{B}'$ such that $\mathcal{M}_{\mathscr{B}'}\subset \mathcal{M}_\mathscr{B}$ and $\nu_{\eta'}=\nu_\eta$.
Keller~\cite{Ke} proved that the Mirsky measure of any taut set has full support. Therefore, whenver $\nu_\eta=\nu_{\eta'}\neq \delta_{(\dots,0,0,0,\dots)}$ then any block of the form $10\dots01$ appearing in $\eta'$ (and there are infinitely many such blocks as we exclude the Dirac measure at $(\dots,0,0,0,\dots)$!) is in fact of positive $\nu_\eta$-measure.  
\end{proof}

           \appendix
\section{Ergodic theory viewpoint}

    \subsection{Direct answer to Question \ref{Q1}\ref{b}}\label{erg4b}
        Let us first recall the remaining necessary notions from ergodic theory and theory of joinings (for more information, we refer the reader, e.g., to~\cite{MR2809170,MR2723325,MR1958753}). Given two measure-preserving transformations $(X_i,\mathcal{B}_i,\mu_i,T_i)$, $i=1,2$, any $\rho\in \mathcal{M}(X_1\times X_2,T_1\times T_2)$ that projects onto $\mu_1$ and $\mu_2$ onto the first and second coordinate, respectively, is called a joining of $T_1$ and $T_2$. The set of joinings is always non-empty (it contains the product measure). If $T_1=T_2$, we speak of self-joinings. The diagonal self-joining of $(X,\mathcal{B},\mu, T)$ is determined by $\triangle(A\times B)=\mu(A\cap B)$ for $A, B\in\mathcal{B}$. If $(Z,\mathcal{D},\rho,R)$ is a common factor of $T_1$ and $T_2$, then also the set of joinings of $T_1$ and $T_2$ that project onto the diagonal self-joining of the common factor is non-empty (it contains the so-called relatively independent extension over the common factor, see~\cite{MR1958753}).

        \begin{prop}\label{p:posent}
           Assume that $\nu,\kappa\in \mathcal{M}^e(\{0,1\}^{\Z},S)$ satisfy $h(\nu)=0$ with $\nu\neq \delta_{(\ldots0,0,0\ldots)}$ and $h(\kappa)>0$. Then $h(\nu\ast\kappa)>0$. \end{prop}
        \begin{proof}
           Consider $(\{0,1\}^{\Z}\times\{0,1\}^\Z,\nu\otimes\kappa,S\times S)$ and denote by $\Pi(\kappa)\subset\mathcal{B}$ the Pinsker $\sigma$-algebra of~$\kappa$. Recall that for $(X_i,\mu_i,T_i)$, $i=1,2$, we have the corresponding relation between the Pinsker $\sigma$-algebras: $\Pi(X_1\times X_2,\mu_1\otimes\mu_2, T_1\times T_2)=\Pi(T_1)\otimes\Pi(T_2)$, see, e.g.\ \cite{GLASNER_2000}. It follows that
           \begin{equation}\label{pinsker1}
              \Pi(\nu\otimes\kappa)=\mathcal{B}\otimes\Pi(\kappa).
           \end{equation}
           Let $C:=\{x\in\{0,1\}^\Z : x_0=1\}$ and suppose that $h(\nu\ast\kappa)=0$, i.e.\ $\Pi(\nu\ast\kappa)=\mathcal{B}$. Therefore, additionally using~\eqref{pinsker1}, we obtain
           \[
              M^{-1}(\mathcal{B})=M^{-1}(\Pi(\nu\ast\kappa))\subset\Pi(\nu\otimes\kappa)=\mathcal{B}\otimes\Pi(\kappa)
           \]
           and it follows that
           $$
              C\times C=M^{-1}{C}\in\mathcal{B}\otimes\Pi(\kappa)
           $$
           (even though $C\times C=M^{-1}C$ is an equality between sets, we think of it up to sets of measure zero, cf.\ also Remark~\ref{ram12}).
           Hence, for $C$ on the second coordinate in $C\times C$, we have $C\in\Pi(\kappa)$.\footnote{
              Let $H_1,H_2$ be Hilbert spaces and let $G_2\subset H_2$ be a closed subspace. Suppose that
              $
                 f\otimes g\in H_1\otimes G_2, \text{ with }f\neq 0.
              $
              Let $g=g_0+g'_0$, with $g_0\in G_2$ and $g'_0\in G_2^\perp$. It follows that $f\otimes g'_0\in H_1\otimes G_2$. But, on the other hand, we can approximate
              $f\otimes g'_0$ by tensors of the form $\sum_{n}\alpha_n f_k\otimes h_k$ with $h_k\in G_2$ which are all orthogonal with $f\otimes g'_0$. This means that $g_0'=0$ and, thus, we have $g\in G_2$.}
           Since $\{C,C^c\}$ is a generating partition, $\Pi(\kappa)=\mathcal{B}$ (modulo $\kappa$) and it follows immediately that $h(\kappa)=0$.
        \end{proof}

    \subsection{Simple proof of Theorem~\ref{tw:postac}}\label{se:a2}

 	We begin this section by the following simple but general observation (it overlaps with Theorem~\ref{tw:postac} for uniquely ergodic $\mathscr{B}$-free systems):

        \begin{prop}\label{simple}
           Suppose that $(Y,S)$ is a uniquely ergodic subshift of $\{0,1\}^\Z$. Let $\widetilde{Y}=M(Y\times \{0,1\}^\Z)$ be the hereditary closure of $Y$. Then, for any $\nu\in \mathcal{M}^e(\widetilde{Y},S)$, there exists $\rho\in \mathcal{M}^e(Y\times \{0,1\}^\Z,S\times S)$ such that $M_\ast(\rho)=\nu$.
        \end{prop}
        \begin{proof}
           Let $z\in \widetilde{Y}$ be a generic point for $\nu$. Then there exists $y\in Y$ such that $z\leq y$. Moreover, $y$ is generic for the unique $S$-invariant measure on $Y$. Let $x\in \{0,1\}^\Z$ be such that $M(y,x)=z$. Notice that $(y,x)$ is quasi-generic for some measure $\rho\in \mathcal{M}(Y\times \{0,1\}^\Z,S\times S)$. Moreover, $M_\ast(\rho)=\nu$ follows directly from the equality $M(y,x)=z$. To complete the proof, it suffices to use the ergodic decomposition of $\rho$ (the image of a convex combination of measures is a convex combination of their images, with the same coefficients).
        \end{proof}

        \begin{rema}
           The original proof of Theorem~\ref{tw:postac} is much more involved than what we present below. However, it includes much more information about the structure of invariant measures for $(\widetilde{X}_\eta,S)$. E.g.\ it serves as a tool to prove that $(\widetilde{X}_\eta,S)$ is intrinsically ergodic~\cite{Ku-Le-We, Dy-Ka-Ku-Le}. Cf.\ also Remark~\ref{rema:a7}.
        \end{rema}

        Let now $\mathscr{B}=\{b_k : k\geq 1\}\subset\mathbb{N}\setminus\{1\}$ and, for each $K\geq 1$, let $\mathscr{B}_K:=\{b_1,\dots, b_K\}$. Set $\eta:=\raz_{\mathcal{F}_\mathscr{B}}$ and $\eta_K:=\raz_{\mathcal{F}_{\mathscr{B}_K}}$. The Mirsky measure $\nu_{\eta_K}$ (considered on $\widetilde{X}_{\eta_K}$) is the purely atomic measure given by the periodic point $\eta_K$. Moreover, $\eta_K$ is a generic point for $\nu_{\eta_K}$ while $\eta$ is quasi-generic for $\nu_\eta$, see~\cite{Dy-Ka-Ku-Le}.  Recall also the following classical result of Davenport and Erd\"os:
        \begin{theo}[\cite{Davenport:1933aa,Davenport1936}]\label{de}
           For any $\mathscr{B}\subset \N\setminus\{1\}$,
           \[
              \underline{d}(\mathcal{M}_\mathscr{B})=\lim_{n\to\infty}d(\mathcal{M}_{\mathscr{B}_K}),
           \]
           ($\underline{d}$ and $d$ denote the lower density and usual asymptotic density, respectively).\footnote{For subset $A\subset\Z$ symmetric with respect to $0$, we have $d(A)=\lim_{N\to\infty}\frac1N|A\cap [1,N]|=\lim_{N\to\infty}\frac1{2N}|A\cap [-N,N]|$; an analogous relation holds for $\underline{d}$.}
        \end{theo}

        \begin{lema}\label{granica}
           For each  $\mathscr{B}=\{b_k : k\geq 1\}\subset \N\setminus \{1\}$,
           $\nu_{\eta_K}\to \nu_\eta$ weakly, as $K\to \infty.$
        \end{lema}
        \begin{proof}
           It suffices to show that
           $$
              \int f\, d\nu_{\eta_K} \to\int f\, d\nu_{\eta}
           $$
           for functions $f$ on $\{0,1\}^\Z$ depending on a finite number (say, $L$) of coordinates. Let $(N_k)_{k\geq 1}$ be an increasing sequence such that $\lim_{k\to \infty}\frac{1}{N_k}|\mathcal{M}_\mathscr{B}\cap [1,N_k]|=\underline{d}(\mathcal{M}_\mathscr{B})$. We then have
           \begin{align*}
              \left|\int f\, d\nu_{\eta_K} -\int f\, d\nu_{\eta}\right| & =\lim_{k\to \infty}\left|\int f\, d \frac{1}{N_k}\sum_{n\leq N_k}\delta_{S^n\eta_K}-\int f\, d\frac{1}{N_k}\sum_{n\leq N_k}\delta_{S^n\eta}\right|  \\
                                                                        & \leq \lim_{k\to \infty}\frac{1}{N_k}\sum_{n\leq N_k}\left|f(S^n\eta_K)-f(S^n\eta)\right|                                                            \\
                                                                        & \leq 2\| f\|\cdot (2L-1)\cdot \lim_{k \to \infty}\frac{1}{N_k}\sum_{n\leq N_k}\left|\{1\leq n\leq N_k : \eta_K(n)\neq \eta(n) \}\right|             \\
                                                                        & =2\|f\|\cdot (2L-1)\cdot \left|\underline{d}(\mathcal{M}_\mathscr{B})-d\left(\mathcal{M}_{\mathscr{B}_K}\right)\right|\to 0 \text{ as }K\to \infty,
           \end{align*}
           where the convergence follows from Theorem~\ref{de}.
        \end{proof}

        \begin{proof}[Proof of Theorem~\ref{tw:postac}]
           Take $\nu\in\mathcal{M}^e(\widetilde{X}_\eta,S)$.
           Since for $K\geq 1$, we have $\eta\leq \eta_K$ (coordinatewise), it follows that $\widetilde{X}_{\eta_K}\supset \widetilde{X}_\eta$, whence $\nu\in\mathcal{M}^e(\widetilde{X}_{\eta_K},S)$. Let $u_K\in \widetilde{X}_{\eta_K}$ be a generic point for $\nu$. Since $u_K\in \widetilde{X}_{\eta_K}$, $u_K\leq  S^i\eta_K$ for some $i$ (because we consider the hereditary closure of a periodic sequence). In other words, we have $u_K=S^i\eta_K \cdot y_K$ for some $y_K\in \{0,1\}^\Z$. We may assume without loss of generality that $i=0$ (since $S^{-i}u_K$ and $u_K$ are generic for the same measure). Now, $(\eta_K,y_K)$ is quasi-generic for a measure ${\rho}_K$ defined on ${X}_{\eta_K}\times \{0,1\}^\Z$. Note that its projection $\rho_K|_{X_{\eta_K}}$ onto the first coordinate satisfies $\rho_K|_{X_{\eta_K}}=\nu_{\eta_K}$. Moreover,
           $$
              M_\ast({\rho}_K)=\nu
           $$
           as $u_K=\eta_K\cdot y_K=M(\eta_K,y_K)$ is (quasi-)generic for $M_\ast({\rho}_K)$ and generic for $\nu$. Passing to a subsequence, if necessary, ${\rho}_K\to {\rho}$ (a measure on $X_\eta\times \{0,1\}^\Z$).
           Therefore, we have
           $$
              \nu=M_\ast({\rho}_K)\to M_\ast({\rho}),$$
           so $\nu=M_\ast({\rho})$. Moreover,
           $$\nu_{\eta_K}=\rho_K |_{X_\eta}\to \rho|_{X_\eta}
           $$
           so $\rho|_{X_\eta}=\nu_\eta$, in view of Lemma~\ref{granica}.
        \end{proof}

    \subsection{Example related to Question~\ref{Q1}\ref{b}}\label{ergo}
        This section is related to Remark~\ref{rema1.8}: it turns out that after relaxing the independence assumption~\ref{DWA}, there might be plenty of joint distributions of $(\mb{X},\mb{Y})$ such that $\h[\mb{X}]>\h[\mb{M}]=0$, with $\mb{X}$ and $\mb{Y}$, satisfying \ref{RAZ}. More precisely, one can prove the following ergodic-theoretic result on $\mathscr{B}$-free systems:

        \begin{theo}\label{pro:bwolne}
	For any $\mathscr{B}\subset \mathbb{N}\setminus\{1\}$, there exists $\rho\in \mathcal{M}(X_\eta\times\{0,1\}^\Z,S\times S)$ with $\rho|_{X_\eta}=\nu_\eta$, such that $h(\rho,S\times S)>0$ and $h(M_\ast(\rho),S)=0$.
        \end{theo}
        \begin{rema}\label{rema:a7}
 The proof of Theorem~\ref{pro:bwolne} is quite technical and it is beyond the scope of this paper, as we put emphasis on the ``independent case''. It will be published elsewhere. We present it below in the simplest possible case, i.e.\ for $\mathscr{B}=\{2\}$. Then $X_\eta=\{a,b\}$, where $a:=(\ldots 0,1,0,1,\underaccent{\circ}{0},1,0,1\ldots)$ and $b:=Sa$, where $S$ is the left shift on $\{0,1\}^\Z$.
            Our approach is ergodic-theoretic and draws heavily on the description of invariant measures for $(\widetilde{X}_\eta,S)$ from~\cite{Ku-Le-We,Dy-Ka-Ku-Le}. The notation is also related to the one in~\cite{Ku-Le-We, Dy-Ka-Ku-Le}.
        \end{rema}
        \begin{prop}\label{A5}
           There exists  $\rho\in \mathcal{M}(\{a,b\}\times \{0,1\}^\Z,S\times S)$ such that $h(\rho,S\times S)>0$, whereas $h(M_\ast(\rho),S)=0$.
        \end{prop}
        Define $
           \overline{\Psi}\colon \{a,b\}\times \{0,1\}^\Z \to \{a,b\}\times \{0,1\}^\Z\times \{0,1\}^\Z$ in the following way:
        \begin{equation*}
           \begin{alignedat}{3}
              \overline{\Psi}(a,x)&=(a,\widehat{x}_a,\widetilde{x}_a),&\text{ where } \widehat{x}_a=(\dots,\underaccent{\circ}{x}_1,x_3,x_5,\dots) \text{ and }&\widetilde{x}_a=(\dots,\underaccent{\circ}{x}_0,x_2,x_4,\dots),\\
              \overline{\Psi}(b,x)&=(b,\widehat{x}_b,\widetilde{x}_b),&\text{ where } \widehat{x}_b=(\dots,\underaccent{\circ}{x}_0,x_2,x_4,\dots) \text{ and }&\widetilde{x}_b=(\dots,\underaccent{\circ}{x}_1,x_3,x_5,\dots).
           \end{alignedat}
        \end{equation*}
        Clearly, $\widetilde{x}_c=\widehat{x}_{Sc}$ for $c\in \{a,b\}$. Notice that one can interpret $\widehat{x}_c$ as ``survivors'', i.e.\ these coordinates of $x$ that ``survive'' after applying $M$ to $(c,x)$, $c\in \{a,b\}$. Similarly, $\widetilde{x}_c$, $c\in \{a,b\}$ can be seen as ``victims'', i.e.\ the coordinates of $x$ that are ``killed'' after applying $M$ to $(c,x)$. Moreover, $\overline{\Psi}$ is a homeomorphism.
        \begin{lema}
           We have $\overline{\Psi}\circ (S\times S) = \overline{S} \circ \overline{\Psi}$, where $$\overline{S}(a,y,z)=(Sa,y,Sz)\text{ and } \overline{S}(b,y,z)=(Sb,Sy,z).$$
        \end{lema}
        \begin{proof}
           Direct calculation.
        \end{proof}
Denote the restriction of $\overline{S}$ to the first to coordinates by $\widehat{S}$ and to the first and third coordinate by $\widetilde{S}$. All these maps are homeomorphisms. Thus, $\overline{S}$ can be viewed as a (topological) joining of $\widehat{S}$ and $\widetilde{S}$. 
Both, $\widehat{S}$ and $\widetilde{S}$ act on $\{a,b\}\times \{0,1\}^\Z$; in this joining, the first coodinates of $\widehat{S}$ and $\widetilde{S}$ are glued diagonally. Each choice of an invariant measure for $\overline{S}$ yields a joining (in ergodic theoretic sense) of the corresponding projections for $\widehat{S}$ and $\widetilde{S}$.

        \begin{lema}
           We have $M=m\circ \pi_{1,2}\circ\overline\Psi$, where $\pi_{1,2}\colon \{a,b\}\times \{0,1\}^\Z \times \{0,1\}^\Z\to \{a,b\}\times\{0,1\}^\mathbb{Z}$ stands for the projection onto the first two coordinates and $m\colon \{a,b\}\times \{0,1\}^\Z\to \widetilde{\{a,b\}}=M(\{a,b\}\times \{0,1\})$ is given by $m(a,y)=(\dots, \underaccent{\circ}{0},y_0,0,y_1,0,y_2,\dots)$, $m(b,y)=(\dots, \underaccent{\circ}{y}_0,0,y_1,0,y_2,0,\dots)$.
        \end{lema}
        \begin{proof}
           Direct calculation.
        \end{proof}
        We can summarize the above in the following commuting diagram ($\pi_{1,3}$ stands for the projection onto the first and third coordinate):
        $$
           \begin{tikzpicture}[baseline=(current  bounding  box.center)]
              \node (A1) at (0,0) {$(\{a,b\}\times \{0,1\}^\Z,S\times S)$};
              \node (A2) at (5,0) {$(\{a,b\}\times \{0,1\}^\Z\times \{0,1\}^\Z,\overline{S})$};
              \node (B1) at (3,-2) {$(\{a,b\}\times \{0,1\}^\Z,\widehat{S})$};
              \node (B2) at (7,-2) {$(\{a,b\}\times \{0,1\}^\Z,\widetilde{S})$};
              \node (C) at (0,-4) {$(\widetilde{\{a,b\}},S)$};
              \draw[->]
              (A1) edge node[auto] {$\overline{\Psi}$} (A2)
              (A2)  edge node[auto] {$\pi_{1,2}$} (B1)
              (A2)  edge node[auto] {$\pi_{1,3}$} (B2)
              (B1) edge node[auto] {$m$} (C)
              (A1) edge node[auto] {$M$} (C);
           \end{tikzpicture}
        $$

        \begin{proof}[Proof of Proposition~\ref{A5}]
        We start with any $\widehat{\kappa}\in \mathcal{M}^e(\{a,b\}\times \{0,1\}^\Z,\widehat{S})$ and $\widetilde{\kappa}\in \mathcal{M}^e(\{a,b\}\times \{0,1\}^\Z,\widetilde{S})$. The projection of both $\widehat{\kappa}$ and $\widetilde{\kappa}$ onto the first coordinate is the unique $S$-invariant measure on $\{a,b\}$, i.e.\ equals $\frac{1}{2}\left(\delta_{a} + \delta_{b}\right)$. Note that this is nothing but $\nu_\eta$ corresponding to $\mathscr{B}=\{2\}$.
           Therefore, we can ``glue'' these coordinates together diagonally and find $\kappa\in \mathcal{M}(\{a,b\}\times \{0,1\}^\Z\times \{0,1\}^\Z,\overline{S})$ such that
           \[
              (\pi_{1,2})_\ast(\kappa)=\widehat{\kappa} \text{ and }(\pi_{1,3})_\ast(\kappa)=\widetilde{\kappa}
           \]
           (for instance, one can take so-called relatively independent extension of the diagonal joining of the first coordinates).

           Now, suppose additionally that $0=h(\widehat{S},{\widehat{\kappa}})<h(\widetilde{S},\widetilde{\kappa})$ (e.g., one can take $\widehat{\kappa}=\nu_\eta\otimes \delta_{(\ldots,0,0,0,\ldots)}$ and $\widetilde{\kappa}=\nu_\eta\otimes B_{\nicefrac12,\nicefrac12}$). Then
           \[
              h(\widetilde{S},\widetilde{\kappa})\leq h(\overline{S},\kappa)\leq h(\widetilde{S},\widetilde{\kappa}) + h(\widehat{S},\widehat{\kappa})=h(\widetilde{S},\widetilde{\kappa}),
           \]
           where the first inequality follows from the fact that $(S,\widetilde{\kappa})$ is a factor of $(\overline{S},\kappa)$, the second one is a direct consequence of $(\overline{S},\kappa)$ being a joining of $(S,\widetilde{\kappa})$ and $(S,\widehat{\kappa})$, and we use $h({\widehat{\kappa}})=0$.
           Thus,
           \[
              h(\overline{S},\kappa)=h(\widetilde{S},\widetilde{\kappa})>0.
           \]
           Moreover,
           $$
              h(S,m_\ast (\widehat{\kappa}))\leq h(\widehat{S},\widehat{\kappa})=0, \text{ whence }h(S,m_\ast (\widehat{\kappa}))=0.
           $$
           Let $\rho:=(\overline{\Psi}^{-1})_\ast(\kappa)$. We obtain
           $$
              h(\rho,S\times S)=h(\overline{S},\kappa)>0.
           $$
           Moreover, $M_\ast(\rho)=(m\circ \pi_{1,2}\circ\overline\Psi)_\ast(\rho)=(m\circ \pi_{1,2})_\ast(\kappa)=m_\ast(\widehat{\kappa})$ and it follows immediately that
           $$
              h(M_\ast(\rho),S)=0.
           $$
        \end{proof}
        \begin{rema}\label{a8}
           Notice that in the above proof $\widehat{\kappa}$ and $\widetilde{\kappa}$ is arbitrary, the only additional assumption was concerned with their entropy. This (together with the fact that $\overline{\Psi}$ is a homeomorphism) indicates that the set of measures $\rho$ satisfying the assertion of Proposition~\ref{A5} is very rich.
        \end{rema}
        
        \paragraph*{Acknowlegements}
Research of the first author is supported by Narodowe Centrum Nauki grant UMO-2019/33/B/ST1/00364.

        \bibliographystyle{siam}
        \allowdisplaybreaks
        \small
        \bibliography{entropy_of_conv.bib}

\begin{thebibliography}{10}

\bibitem{MR876396}
{\sc H.~Berbee}, {\em Periodicity and absolute regularity}, Israel J. Math., 55
  (1986), pp.~289--304.

\bibitem{MR1512943}
{\sc A.~S. Besicovitch}, {\em On the density of certain sequences of integers},
  Math. Ann., 110 (1935), pp.~336--341.

\bibitem{Markovtail}
{\sc D.~Blackwell and D.~Freedman}, {\em The tail $\sigma$-field of a {M}arkov
  chain and a theorem of {O}rey}, The Annals of Mathematical Statistics, 35
  (1964), pp.~1291--1295.

\bibitem{MR2325294}
{\sc R.~C. Bradley}, {\em Introduction to strong mixing conditions. {V}ol. 1},
  Kendrick Press, Heber City, UT, 2007.

\bibitem{zbMATH03014412}
{\sc S.~Chowla}, {\em On abundant numbers}, J. Indian Math. Soc., New Ser., 1
  (1934), pp.~41--44.

\bibitem{chung1967markov}
{\sc K.~L. Chung}, {\em Markov Chains with Stationary Transition Probabilities:
  2d Ed}, Springer, 1967.

\bibitem{Davenport:1933aa}
{\sc H.~Davenport}, {\em {\"U}ber numeri abundantes},
  Sitzungsber.Preuss.Akad.Wiss.,  (1933), pp.~830--837.

\bibitem{Davenport1936}
{\sc H.~Davenport and P.~Erd\"os}, {\em On sequences of positive integers},
  Acta Arithmetica, 2 (1936), pp.~147--151.

\bibitem{MR0043835}
\leavevmode\vrule height 2pt depth -1.6pt width 23pt, {\em On sequences of
  positive integers}, J. Indian Math. Soc. (N.S.), 15 (1951), pp.~19--24.

\bibitem{finetti31}
{\sc B.~{de Finetti}}, {\em Funzione caratteristica di un fenomeno aleatorio},
  Atti della R. Accademia Nazionale dei Lincei, Ser. 6. Memorie, Classe di
  Scienze Fisiche, Matematiche e Naturali 4,  (1931), pp.~251--299.

\bibitem{MR2809170}
{\sc T.~Downarowicz}, {\em Entropy in {D}ynamical {S}ystems}, vol.~18 of New
  Mathematical Monographs, Cambridge University Press, Cambridge, 2011.

\bibitem{Dy-Ka-Ku-Le}
{\sc A.~Dymek, S.~Kasjan, J.~Ku\l{}aga-Przymus, and M.~Lema\'{n}czyk}, {\em
  {$\mathscr{B}$}-free sets and dynamics}, Trans. Amer. Math. Soc., 370 (2018),
  pp.~5425--5489.

\bibitem{MR2723325}
{\sc M.~Einsiedler and T.~Ward}, {\em Ergodic theory with a view towards number
  theory}, vol.~259 of Graduate Texts in Mathematics, Springer-Verlag London,
  Ltd., London, 2011.

\bibitem{MR1574879}
{\sc P.~Erd\"os}, {\em On the {D}ensity of the {A}bundant {N}umbers}, J. London
  Math. Soc., 9 (1934), pp.~278--282.

\bibitem{feller2008introduction}
{\sc W.~Feller}, {\em An introduction to probability theory and its
  applications}, vol.~1, John Wiley \& Sons, 2008.

\bibitem{MR274718}
{\sc N.~A. Friedman and D.~S. Ornstein}, {\em On isomorphism of weak
  {B}ernoulli transformations}, Advances in Math., 5 (1970), pp.~365--394
  (1970).

\bibitem{MR0213508}
{\sc H.~Furstenberg}, {\em Disjointness in ergodic theory, minimal sets, and a
  problem in {D}iophantine approximation}, Math. Systems Theory, 1 (1967),
  pp.~1--49.

\bibitem{Fu-Pe-We}
{\sc H.~Furstenberg, Y.~Peres, and B.~Weiss}, {\em Perfect filtering and double
  disjointness}, Annales de l'I.H.P. Probabilit\'es et statistiques, 31 (1995),
  pp.~453--465.

\bibitem{MR2783975}
{\sc R.~Garbit}, {\em A note on {F}urstenberg's filtering problem}, Israel J.
  Math., 182 (2011), pp.~333--336.

\bibitem{MR1958753}
{\sc E.~Glasner}, {\em Ergodic theory via joinings}, vol.~101 of Mathematical
  Surveys and Monographs, American Mathematical Society, Providence, RI, 2003.

\bibitem{GLASNER_2000}
{\sc E.~Glasner, J.-P. Thouvenot, and B.~Weiss}, {\em Entropy theory without a
  past}, Ergodic Theory and Dynamical Systems, 20 (2000), pp.~1355--1370.

\bibitem{MR3134681}
{\sc R.~M. Gray}, {\em Entropy and information theory}, Springer, New York,
  second~ed., 2011.

\bibitem{Ha}
{\sc R.~R. Hall}, {\em Sets of multiples}, vol.~118 of Cambridge Tracts in
  Mathematics, Cambridge University Press, Cambridge, 1996.

\bibitem{MR76206}
{\sc E.~Hewitt and L.~J. Savage}, {\em Symmetric measures on {C}artesian
  products}, Trans. Amer. Math. Soc., 80 (1955), pp.~470--501.

\bibitem{Hochman-notes}
{\sc M.~Hochman}, {\em Lectures on dynamical systems and entropy}.
\newblock
  \url{http://math.huji.ac.il/~mhochman/courses/dynamics2014/notes.5.pdf}.

\bibitem{Ke}
{\sc G.~Keller}, {\em Tautness of sets of multiples and applications to
  $\mathcal{B}$-free systems}, Studia Math., 247 (2019), pp.~205--216.

\bibitem{MR259068}
{\sc W.~Krieger}, {\em On entropy and generators of measure-preserving
  transformations}, Trans. Amer. Math. Soc., 149 (1970), pp.~453--464.

\bibitem{Ku-Le-We}
{\sc J.~Ku{\l}aga-Przymus, M.~Lema{{\'n}}czyk, and B.~Weiss}, {\em On invariant
  measures for $\mathscr{B}$-free systems}, Proc. Lond. Math. Soc. (3), 110
  (2015), pp.~1435--1474.

\bibitem{Maker_1940}
{\sc P.~T. Maker}, {\em The ergodic theorem for a sequence of functions}, Duke
  Mathematical Journal, 6 (1940), pp.~27--30.

\bibitem{MR288797}
{\sc R.~A. Olshen}, {\em The coincidence of measure algebras under an
  exchangeable probability}, Z. Wahrscheinlichkeitstheorie und Verw. Gebiete,
  18 (1971), pp.~153--158.

\bibitem{MR346132}
{\sc D.~S. Ornstein and B.~Weiss}, {\em Finitely determined implies very weak
  {B}ernoulli}, Israel J. Math., 17 (1974), pp.~94--104.

\bibitem{Ornstein_1975}
\leavevmode\vrule height 2pt depth -1.6pt width 23pt, {\em Every transformation
  is bilaterally deterministic}, Israel Journal of Mathematics, 21 (1975),
  pp.~154--158.

\bibitem{MR3430278}
{\sc R.~Peckner}, {\em Uniqueness of the measure of maximal entropy for the
  squarefree flow}, Israel J. Math., 210 (2015), pp.~335--357.

\bibitem{MR0152628}
{\sc M.~S. Pinsker}, {\em Dynamical systems with completely positive or zero
  entropy}, Soviet Math. Dokl., 1 (1960), pp.~937--938.

\bibitem{sarnak-lectures}
{\sc P.~Sarnak}, {\em Three lectures on the {M}{\"o}bius function, randomness
  and dynamics}.
\newblock \url{http://publications.ias.edu/sarnak/}.

\bibitem{MR0442198}
{\sc P.~Shields}, {\em The theory of {B}ernoulli shifts}, The University of
  Chicago Press, Chicago, Ill.-London, 1973.
\newblock Chicago Lectures in Mathematics.

\bibitem{MR121856}
{\sc V.~A. Volkonski\u{\i} and Y.~A. Rozanov}, {\em Some limit theorems for
  random functions. {I}}, Theor. Probability Appl., 4 (1959), pp.~178--197.

\bibitem{MR0137141}
\leavevmode\vrule height 2pt depth -1.6pt width 23pt, {\em Some limit theorems
  for random functions. {II}}, Teor. Verojatnost. i Primenen., 6 (1961),
  pp.~202--215.

\end{thebibliography}

        \bigskip
        \footnotesize
        \noindent
        Joanna Ku\l aga-Przymus\\
        \textsc{Faculty of Mathematics and Computer Science, Nicolaus Copernicus University,\\ Chopina 12/18, 87-100 Toru\'{n}, Poland}\par\nopagebreak
        \noindent
        \href{mailto:joanna.kulaga@gmail.com}
        {\texttt{joanna.kulaga@gmail.com}}

        \medskip

        \noindent
        Micha\l{} D.\ Lema\'{n}czyk \\
        \textsc{Faculty of Mathematics, Informatics and Mechanics, Warsaw University,\\ Stefana Banacha 2, 02-097 Warsaw, Poland}\par\nopagebreak
        \noindent
        \href{mailto:m.lemanczyk@mimuw.edu.pl}
        {\texttt{m.lemanczyk@mimuw.edu.pl}}
\end{document}